\documentclass[reqno,11pt]{amsart}
\usepackage[english]{babel} 
\usepackage{hyperref,amsthm}
\hypersetup{colorlinks=true,linkcolor=blue,urlcolor=blue} 
\usepackage[hyperpageref]{backref}
\usepackage{hyperref}
\usepackage[mathscr]{eucal}
\usepackage{amssymb,amsmath,amsthm}
\usepackage{frcursive}
\usepackage{calrsfs}
\newtheorem{theorem}{Theorem}[section]
\newtheorem{lemma}[theorem]{Lemma}
\newtheorem{corollary}[theorem]{Corollary}
\newtheorem{proposition}[theorem]{Proposition}

\newtheorem{conjecture}[theorem]{Conjecture}
\newtheorem{remark}[theorem]{Remark}

\def\Q{\mathbb{Q}}
\def\Z{\mathbb{Z}}
\def\F{\mathbb{F}}
\let\ds=\displaystyle
\let\sst=\scriptstyle
\let\ov=\overline

\let\wt=\widetilde
\let\wh=\widehat
\let\ft=\footnotesize
\def\No{{\rm N}}
\def\G{{\rm Gal}}
\def\Sha{{\rm III}}

\def\order{\raise1.5pt \hbox{${\scriptscriptstyle \#}$}}
\def\lien{\mathrel{\mkern-4mu}}
\def\too{\relbar\lien\rightarrow}
\def\tooo{\relbar\lien\relbar\lien\too}

\newcommand{\plus}{\ds\mathop{\raise 2.0pt \hbox{$\bigoplus $}}\limits}
\newcommand{\prd}{\ds\mathop{\raise 2.0pt \hbox{$  \prod   $}}\limits}
\newcommand{\sm}{\ds\mathop{\raise 2.0pt \hbox{$  \sum    $}}\limits}
\def\ffrac#1#2{\hbox{\footnotesize $\displaystyle \frac{#1}{#2}$}}
\newcommand{\limproj}{\mathop{\oalign{{\rm lim}\cr\hidewidth$\longleftarrow$\hidewidth\cr}}}

\begin{document}

\markboth{Georges Gras}{}

\title[Weber's class number problem and $p$--rationality]
{Weber's class number problem and $p$--rationality \\
in the cyclotomic $\wh \Z$--extension of $\Q$}

\author{Georges Gras}
\address{Villa la Gardette, 4 Chemin Ch\^ateau Gagni\`ere,
F-38520 Le Bourg d'Oisans}
\email{g.mn.gras@wanadoo.fr}

\begin{abstract} Let $K$ be the $N$th layer in the cyclotomic 
$\wh \Z$-extension $\wh \Q$. Many authors (Aoki, Fukuda, Horie, Ichimura, Inatomi, 
Komatsu, Miller, Morisawa, Nakajima, Okazaki, Washington,\,$\ldots$) 
analyse the behavior of the $p$-class groups ${\mathcal C}_K$. 
We revisit this problem, in a more conceptual form, since
computations show that the $p$-torsion group ${\mathcal T}_K$ 
of the Galois groups of the maximal abelian $p$-ramified 
pro-$p$-extension of $K$ (Tate--Shafarevich group of $K$) 
is often non-trivial; this raises questions since 
$\order {\mathcal T}_K = \order {\mathcal C}_K\, \order {\mathcal R}_K$
where ${\mathcal R}_K$ is the normalized $p$-adic regulator.
We give a new method testing ${\mathcal T}_K \ne 1$ (Theorem \ref{st*},
Table \ref{table}) and characterize the $p$-extensions $F$ of $K$ in $\wh \Q$ 
with ${\mathcal C}_F \ne 1$ (Theorem \ref{thmfond} and Corollary \ref{eight}). 
We publish easy to use programs, justifying again the eight known examples, 
and allowing further extensive computations.
\end{abstract}

\date{November 5,  2020}

\keywords{$p$-class groups, cyclotomic $\Z_\ell$-extensions,
class field theory, $p$-adic regulators, $p$-ramification theory, PARI/GP 
programs}
\subjclass{11R29, 11R37, 11Y40}

\maketitle

\tableofcontents

\section  {Introduction}
Let $\ell \geq 2$ be a prime number and let $\Q(\ell^n)$, $n \geq 0$, 
be the $n$th layer of the cyclotomic $\Z_\ell$-extension $\Q(\ell^\infty)$ of $\Q$ 
(with $[\Q(\ell^n) : \Q]=\ell^n$). We draw attention on the fact that we use
$\ell$ (instead of $p$ in the literature) since we need to apply the
$p$-ramification theory to the fields $\Q(\ell^n)$, $p \ne \ell$, which 
is more convenient for our presentation and bibliographic references.
The purpose of our study is to see in what circumstances the $p$-class group
of $\Q(\ell^n)$ (then of any composite $\Q(N)$ of such fields) is likely to be non-trivial.

\smallskip
Indeed, one may ask if the arithmetic of these fields is as 
smooth as it is conjectured (for the class group ${\mathcal C}_{\Q(N)}$) 
by many authors after many verifications and partial proofs \cite{AF,BPR,FK1,
FK2,FK22,FK3,FKM,Fu,Ho1,Ho2,Ho3,Ho4,Ich,IcNa,Ina,Mi1,Mi2,
Mor1,Mor2,Mor3,MorOk1,MorOk2,Wa}. 
The triviality of ${\mathcal C}_{\Q(\ell^n)}$ has, so far, no counterexamples 
as $\ell$, $n$, $p$ vary, but that of the Tate--Shafarevich group
${\mathcal T}_{\Q(\ell^n)}$ (or more generally ${\mathcal T}_{\Q(N)}$)
is, on the contrary, not true as we shall see 
numerically, and, for composite $N$, few ${\mathcal C}_{\Q(N)} \ne 1$
have been discovered.

\subsection{History of the main progress and results}
The computation of the class number have 
been done in few cases because of limitation of the order of magnitude of the 
degree $N$ and of $p$; for instance, the results given in \cite[Tables 1, 2]{Mi1,Mi2} 
only concern $\ell^n=2^7$, $3^4$, $5^2$, $11$, $13$, $17$, $19$, $23$, $29$, $31$
($2^7$, $3^4$, $29$, $31$ under GRH). From PARI/GP \cite{PARI}, in
a straightforward use, any 
computation needs the instruction ${\sf bnfinit(P)}$ (giving all the 
basic invariants of the field $K$ defined via the polynomial ${\sf P}$, 
whence the class group, a system of units, etc.); thus, by this way,
few values of $N$ may be carried out. 

\smallskip
Approaches, by means of geometry of numbers, prove
that some of these fields are euclidean (see, e.g., \cite{Ce} 
about $\Q(2^2)$, $\Q(2^3)$); but this more difficult and broad 
aspect, needs other techniques and is hopeless for our goal. 

\smallskip
Then some deep analytic studies of the class number were done 
by many authors (Aoki, Fukuda, Horie, Ichimura, Inatomi, Komatsu, Miller, 
Morisawa, Nakajima, Okazaki, Washington\,$\ldots$) proving infinitely many 
cases of $p$-principality, high enough in the towers.
New PARI functions, for abelian arithmetic, may be available
(see \cite{Fu1} for more information) and deal with classical
analytic formulas (cyclotomic units, Bernoulli numbers, etc.).

\subsection{Method and main results}
Let $K$ be any real abelian number field and let ${\mathcal T}_K$ be the 
torsion group of ${\mathcal G}_K := \G(H_K^{\rm pr}/K)$, where 
$H_K^{\rm pr}$ is the maximal abelian $p$-ramified (i.e., unramified outside 
$p$ and~$\infty$) pro-$p$-extension of $K$; this group is essentially the so-called 
Tate--Shafarevich group. 
The new aspects of our method is the use of $p$-adic measures which are 
more naturally attached to ${\mathcal T}_K$ or to the Jaulent logarithmic class 
groups $\wt {\mathcal C}_K$ related to Greenberg's conjecture.

\smallskip
Curiously, the computation of ${\mathcal T}_K$ is easier than that of 
${\mathcal C}_K$ or of the normalized $p$-adic regulator ${\mathcal R}_K$, 
separately. 

\smallskip
Let  $\wh \Q$ be the composite of the cyclotomic $\Z_\ell$-extensions 
of $\Q$. For $K=\Q(\ell^n)$, or more generally, for the subfields 
$\Q(\ell_1^{n_1}) \cdots \Q(\ell_t^{n_t}) \in \wh \Q$,
(denoted $\Q(N)$ where $N:=\ell_1^{n_1} \cdots \ell_t^{n_t}$), 
we have the identity:
$$\order {\mathcal T}_K = \order {\mathcal C}_K \cdot 
\order {\mathcal R}_K \cdot \order {\mathcal W}_K, $$
where ${\mathcal C}_K$ is the $p$-class group, ${\mathcal R}_K$ the 
normalized $p$-adic regulator, ${\mathcal W}_K=1$ for $p>2$ and 
${\mathcal W}_K \simeq \F_2^{\order S-1}$ for $p=2$, where 
$S := \{{\mathfrak p},\  {\mathfrak p} \mid 2\ \hbox{in $K$}\}$ 
(Lemma \ref{lemmaW}).
Since Leopoldt's conjecture holds in abelian fields, we have, for any prime~$p$,
${\mathcal G}_K = \Gamma_K \oplus {\mathcal T}_K$ with
$\Gamma_K \simeq \Z_p$. 

\smallskip
So, as soon as ${\mathcal T}_K=1$ (i.e., $K$ is called $p$-rational), 
we are certain that ${\mathcal C}_K = 1$; 
otherwise, we may suspect a possible counterexample. 
We shall first compute \S\,\ref{nontrivial} the structure of some 
${\mathcal T}_K$ by means of an indisputable reference program
(but using ${\sf bnfinit(P)}$) to show that this $p$-torsion group is non-trivial in 
many cases. The good new is that there exists a 
test for ${\mathcal T}_K \ne 1$ which does not need ${\sf bnfinit(P)}$ 
and allows large $K$ and $p$, so that our programs are very elementary and 
written with basic instructions giving simpler faster computations in larger intervals; 
it will be explained Sections \ref{annihilation}, \ref{slnc}, and yields Theorem \ref{st*}
and Table \ref{table}.

\smallskip
Finally, we give programs to search non-trivial 
$p$-class groups using Che\-valley's formula in $p$-extensions of $K$
in $\wh \Q$, in connection with a deep property of the $p$-adic regulator 
${\mathcal R}_K$, as product of the form ${\mathcal R}_K = 
{\mathcal R}_K^{\rm nr} \cdot {\mathcal R}_K^{\rm ram}$ 
(diagrams of \S\,\ref{diagram}), in relationship with studies 
of Taya \cite{Ta} on Greenberg's conjecture. We prove (Lemma \ref{totdec})
that, without restricting the generality, one may assume $p$ totally split in $K$,
then ${\mathcal R}_K^{\rm ram}=1$ (Lemma \ref{split}) and (Lemma \ref{ta}) 
that  ${\mathcal R}_K \ne 1$ is equivalent to ${\mathcal C}_{K_m} \ne 1$ for 
$K_m:=K\Q(p^m)$, $m$ large enough, hence the following fundamental 
criterion (Theorem \ref{thmfond}):

\smallskip\noindent
{\bf Main Theorem.} {\it Let  $\wh \Q$ be the composite of all the cyclotomic 
$\Z_\ell$-extensions of $\Q$ and let $K=\Q(N) \subset \wh \Q$. Let $p > 2$ 
be a prime, totally split in $K$. 
Then, ${\mathcal C}_{K \Q(p^m)} = 1$ for all $m \geq 0$, if and only if 
${\mathcal T}_K = 1$ (i.e., $K$ is $p$-rational). For $p=2$, the condition 
becomes ${\mathcal T}_K={\mathcal W}_K$ (i.e., $\order {\mathcal T}_K=2^{N-1}$)}.

\smallskip
Despite of the huge intervals tested, we only find again 
(see Corollary \ref{eight}) known cases (Fukuda--Komatsu--Horie--Morisawa), 
but with programs that may be used by anyone on more powerful computers 
than ours. This suggests an extreme rarity of non-trivial class groups in $\wh \Q$.

\subsection{The $p$-torsion groups ${\mathcal T}_K$ in number theory}
\smallskip
These invariants were less (numerically) computed than 
class groups, which is unfortunate because they are of basic 
significance in Galois cohomology since for all number fields $K$ 
(under Leopoldt's conjecture), ${\mathcal T}_K$ is the dual of 
${\rm H}^2({\mathcal G}_K,\Z_p)$ \cite{Ng}, where ${\mathcal G}_K$ 
is the Galois group of the maximal $p$-ramified pro-$p$-extension of $K$
(ordinary sense); the freeness of ${\mathcal G}_K$ is then equivalent to 
${\mathcal T}_K=1$.
Then, after the pioneering works of Haberland--Koch--Neumann--Schmidt
and others, we have the local-global principle defining first and second 
Tate--Shafarevich groups in the framework of $S$-ramification 
when $S$ is the set of $p$-places of $K$ \cite[Theorem 3.74]{Ko}:
$$\Sha^i_K := {\rm Ker} \Big [{\rm H}^i({\mathcal G}_K,\F_p) \too
\plus_{v \in S}{\rm H}^i({\mathcal G}_{K_v},\F_p) \Big],\ i=1,2; $$ 
then $\Sha^1_K \simeq {\mathcal C}_K/cl_K(S)$ (the $S$-class group),
and $\Sha^2_K$ closely depends on
$V_K := \{\alpha \in K^\times, \, (\alpha) = {\mathfrak a}^p, \ \alpha \in
K_v^{\times p}, \, \forall v \in S\}$, via the exact sequence:
$$0 \to V_K/K^{\times p} \too {\rm H}^2({\mathcal G}_K,\F_p) \too
\plus_{v \in S}{\rm H}^2({\mathcal G}_{K_v},\F_p) \too \Z/p\Z\ ({\rm resp.}\ 0) \to 0,$$
if $\mu_p \subset K$ (resp. $\mu_p\not \subset K$). Finally, the link with the 
invariant ${\mathcal T}_K$ is given by the rank formula
${\rm rk}_p({\mathcal T}_K) = {\rm rk}_p(V_K/K^{\times p})
+ \sm_{v \in S} \delta_v - \delta_K$, where $\delta_v$ (resp. $\delta_K$)
is $1$ or $0$ according as $K_v$ (resp. $K$) contains $\mu_p$ or not 
\cite[Corollary III.4.2.3]{Gra0}, thus giving for $K \subset \wh \Q$ (real fields):
$${\rm rk}_p({\mathcal T}_K) = {\rm rk}_p(V_K/K^{\times p}) 
+ \delta_{p,2}\,(\order S - 1), $$
with an exceptional term, only when $p=2$.
Thus, ${\rm rk}_p(\Sha^2_K)$ is essentially ${\rm rk}_p({\mathcal T}_K)$.
For generalizations, with ramification and decomposition giving
Shafarevich formula, see \cite[II.5.4.1]{Gra0} as well as \cite{Jau0}, 
and for the reflection theorem on generalized class groups, see
\cite[II.5.4.5 and Theorem III.4.2]{Gra0}.

\smallskip
If one replaces the notion of $p$-ramification (in pro-$p$-extensions) by that 
of $\Sigma$-ramification (in pro-extensions), {\it for any set of places} $\Sigma$, 
the corresponding Tate--Shafarevich groups have some relations with the 
corresponding torsion groups ${\mathcal T}_{K,\Sigma}$, but with many open 
questions and interesting phenomena when no assumption is done (see, e.g., 
\cite{HMR1,HMR2} for an up to date story about them and for numerical examples).

\smallskip
When ${\mathcal T}_K=1$ under Leopoldt's conjecture
(freeness of ${\mathcal G}_K$), one speaks of 
$p$-rational field $K$; in this case, the Tate--Shafarevich groups are
trivial or obvious, which has deep consequences as shown for 
instance in \cite{BGKK} in relation with our conjectures in \cite{Gra4}
on the $p$-adic properties of the units. 
For more information on the story of abelian $p$-ramification and for that
of $p$-rationality, see \cite[Appendix~A]{Gra9} and its bibliography about the  
pioneering contributions: ${\rm K}$-theory approach \cite{Gra3},
$p$-infinitesimal approach \cite{Jau0}, cohomological/pro-$p$-group 
approach \cite{Mov,MovN}. All basic material about $p$-rationality
is overviewed in \cite[III.2, IV.3, IV.4.8]{Gra0}.

\smallskip
The orders and annihilations of the ${\mathcal T}_K$ are given by $p$-adic 
$L$-functions ``at $s=1$'', the two theories (arithmetic and analytic) being 
equivalent and given by suitable $p$-adic pseudo-measures
(see Theorem \ref{st*}).

\smallskip
All these principles on Tate--Shafarevich groups exist for the theory 
of elliptic curves and other contexts as the arithmetic of abelian
varieties over the composite of $\Z_\ell$-extensions of a fixed number 
field $F$, the case of $\wh \Q$ ($F=\Q$) being at the origin of a question of 
Coates \cite[Section 3]{Co2} on the possible triviality of the $C_{\Q(N)}$ in $\wh \Q$, 
or at least of their extreme rarity.\,\footnote{\,I warmly thank John Coates for 
sending me his conference paper (loc.cit.), not so easy to find for me, 
but which contains very useful numerical and bibliographical information.}

\subsection{The logarithmic class group and Greenberg's conjecture}
We shall also consider another $p$-adic invariant, the Jaulent's logarithmic 
class group $\wt {\mathcal C}_K$ \cite{Jau1} which governs
Greenberg's conjecture \cite{Gree} for totally real number fields $K$ 
(i.e., $\lambda = \mu = 0$ for the cyclotomic $\Z_p$-extension of $K$), 
the result being that Greenberg's conjecture holds if and only 
$\wt {\mathcal C}_K$ capitulates in $K_\infty := K \Q(p^\infty)$ \cite{Jau2}. 
Of course Greenberg's conjecture holds for $p = \ell$ in 
$\Q(\ell^\infty)$ for trivial reasons, but we have few information for the 
cyclotomic $\Z_p$-extensions of $K=\Q(N) \subset \wh \Q$.
As we shall see, in all attempts concerning subfields of $\wh \Q$,
Jaulent's logarithmic class group was trivial (but the instruction 
${\sf bnflog(P,p)}$ needs ${\sf bnfinit(P)}$ limiting the intervals).
See \S\,\ref{Clog} for more numerical information. In particular,
${\mathcal C}_K={\mathcal R}_K^{\rm nr}=1$ implies 
$\lambda = \mu = \nu = 0$ \cite[Theorem 5]{Gra10}.

\section{Abelian $p$-ramification theory}

Recall the context of abelian $p$-ramification theory when $K$ is any totally real 
number field (under Leopoldt's conjecture for $p$ in $K$).

\subsection{Main definitions and notations}

\smallskip
\begin{enumerate}\label{definitions}
 \item[$(a)$] Let $E_K$ be the group of $p$-principal global units 
$\varepsilon \equiv 1\!\! \pmod {\prod_{{\mathfrak p} \mid p} {\mathfrak p}}$ 
of~$K$. Let $U_K := \hbox{$\oplus_{{\mathfrak p} \mid p}$} 
U_{\mathfrak p}$ 
be the $\Z_p$-module of $p$-principal local units, where 
$U_{\mathfrak p}$ 
is the group of ${\mathfrak p}$-principal units of the 
${\mathfrak p}$-completion $K_{\mathfrak p}$ of~$K$. 
Let $\mu_K$ (resp. $\mu_{\mathfrak p}$) be the group of $p$th 
roots of unity of $K$ (resp. $K_{\mathfrak p}$).

\smallskip
\item[$(b)$] 
Let $\iota : \{x \in K^\times \otimes \Z_p, \, \hbox{$x$ prime to $p$}\} \to
U_K$ be the diagonal embedding. Let $\ov E_K = \iota(E_K \otimes \Z_p)$ be the closure of 
$\iota E_K$ in $U_K$ and let $H_K^{\rm nr}$ be the $p$-Hilbert class field of $K$; 
then we have $\G(H_K^{\rm pr}/H_K^{\rm nr}) \simeq U_K/\ov  E_K$.
The Leopoldt conjecture leads to the (not so trivial) exact sequence:
$$\hspace{0.4cm}1 \to {\mathcal W}_K \too  
{\rm tor}_{\Z_p}^{} \big(U_K \big /\ov E_K \big) 
 \mathop {\tooo}^{{\rm log}}  {\rm tor}_{\Z_p}^{}\big({\rm log}\big 
(U_K \big) \big / {\rm log} (\ov E_K) \big) \to 0, $$

where ${\mathcal W}_K := {\rm tor}_{\Z_p}(U_K)/\iota \mu_K = 
\big[ \oplus_{{\mathfrak p} \mid p} \mu_{\mathfrak p} \big] \big /\iota \mu_K$. 

\item[$(c)$] 
Let ${\mathcal C}_K \simeq \G(H_K^{\rm nr}/K)$ be the 
$p$-class group of $K$ and let $\wt {\mathcal C}_K$ be the logarithmic
class group, isomorphic to $\G(H_K^{\rm lc}/K_\infty)$, where
$H_K^{\rm lc} \subseteq H_K^{\rm pr}$ is the maximal abelian locally 
cyclotomic pro-$p$-extension of $K$.

\smallskip
\item[$(d)$] Let ${\mathcal R}_K := {\rm tor}_{\Z_p} 
({\rm log}(U_K)/{\rm log}(\ov E_K))$
be the normalized $p$-adic regulator \cite[\S\,5]{Gra7}; recall that
for $p\ne 2$, $\order{\mathcal R}_K = \frac{R_K}{p^{d-1}}$ and 
$\order{\mathcal R}_K = \frac{1}{2^{s_2-1}} \frac{R_K}{2^{d-1}}$ for
$p=2$, where $R_K$ is the classical $p$-adic regulator, $d = [K : \Q]$
and $s_2$ is the number of $2$-places in $K$ (see \cite[Appendix]{Co1}
giving the link of ${\mathcal R}_K$ with the residue of the $p$-adic 
zeta function of $K$).

\smallskip
\item[$(e)$] Let $K_\infty := K\Q(p^\infty)$ be the cyclotomic $\Z_p$-extension of $K$
and let $H_K^{\rm bp}$ (called the Bertrandias--Payan field) fixed by 
the subgroup ${\mathcal W}_K$ of ${\mathcal T}_K$; the field $H_K^{\rm bp}$
is the composite of all $p$-cyclic extensions of $K$ embeddable in $p$-cyclic 
extensions of arbitrary large degree.
\end{enumerate}

\subsection{The case of the fields $K= \Q(N)$}
In that case, $K_\infty \cap H_K^{\rm nr} = K$ and ${\mathcal W}_K$ is given as follows:

\begin{lemma}  \label{lemmaW}
One has ${\mathcal W}_K=1$ for 
$K= \Q(N)$, except for $p=2$ in which case, 
${\mathcal W}_K \simeq \F_2^{\order S-1}$ where $S$
is the set of primes ${\mathfrak p} \mid 2$ in $K$.
\end{lemma}

\begin{proof}
For $p \ne 2$, $K_{\mathfrak p}$ does not contain 
$\mu_p$ since $\Q_p(\mu_p)/\Q_p$, of degree $p-1 >1$,
is totally ramified and not contained in $\Q_p(p^\infty)$; 
thus ${\mathcal W}_K=1$.
For $p=2$, $K_{\mathfrak p}$ does not contain $\mu_4$
and ${\rm tor}_{\Z_p}(U_K) \simeq \F_2^{\order S}$, thus 
${\mathcal W}_K \simeq \F_2^{\order S-1}$.
\end{proof}

The case $\order S >1$ is very rare and occurs when $2^{\ell-1} \equiv 1 \pmod {\ell^2}$
for some $\ell \mid N$, e.g., $\ell=2093$, $3511$ which are out 
of range of practical computations. Thus ${\mathcal W}_K$ is in general trivial.
If moreover ${\mathcal C}_K = 1$, ${\mathcal T}_K = {\mathcal R}_K$, 
which is not always trivial as we shall see, even if we have conjectured in \cite{Gra4} 
that, for any number field $K$, ${\mathcal T}_K=1$ for $p \gg 0$.

\subsection{Fixed points formulas}

Let $C_K$ be the whole class group of a number field $K$ 
(in the restricted or ordinary sense, which will be specified with 
the mentions ${}^{\rm res}$ or ${}^{\rm ord}$).

\smallskip
Chevalley's formula \cite[p. 406]{Che} for class groups $C_K^{\rm res}$ 
and $C_k^{\rm res}$, in any cyclic extension $K/k$ of number fields,
of Galois group $G$, is given, in whole generality, by
$\order (C_K^{\rm res})^G = 
 \ffrac{\order C_k^{\rm res} \cdot \prod_{\mathfrak l} e_{\mathfrak l}}
{[K : k] \cdot (E_k^{\rm pos} : E_k^{\rm pos} \cap \No_{K/k} (K^\times))}$, 
where $e_{\mathfrak l}$ is the ramification index in $K/k$ of the prime ideal 
${\mathfrak l}$ of $k$ and $E_k^{\rm pos}$ is the group of 
totally positive units of $k$. When $K/k$ is {\it totally ramified} at some prime 
ideal ${\mathfrak l}_0$, the formula becomes
$\order (C_K^{\rm res})^G  = \order C_k^{\rm res} \cdot 
\ffrac{\prod_{{\mathfrak l} \ne {\mathfrak l}_0}  e_{\mathfrak l}}
{ (E_k^{\rm pos} : E_k^{\rm pos} \cap \No_{K/k} (K^\times))}$
(product of two integers). 

Applied to $\Q(\ell^n)/\Q$ the formula gives $(C_{\Q(\ell^n)}^{\rm res})^G = 1$ 
since $\ell$ is the unique (totally) ramified prime and since $E_\Q^{\rm pos} = 1$.
So, for $p=\ell$, ${\mathcal C}_{\Q(\ell^n)}^{\rm res}=1$, a classical result. 
This fixed points formula is often attributed to Iwasawa (1956), Yokoi (1967)
or others, instead of Chevalley (1933) (more precisely Herbrand--Chevalley, 
the ``Herbrand quotient'' of the group of units of $K$ being the key for a 
general proof). 
The analogous ``fixed points formula'' for the {\it $\ell$-torsion group}
${\mathcal T}_{\Q(\ell^n)}$, in $\Q(\ell^n)/\Q$ gives also 
${\mathcal T}_{\Q(\ell^n)} = 1$
 (\cite[Theorem IV.3.3]{Gra0}, \cite[Proposition~6]{Gra3}, 
\cite[Appendix A.4.2]{Gra9}); which justifies once again the fact that the 
notation ${\mathcal T}$ always refers to a $p$-torsion group. 

\smallskip
Nevertheless,
Chevalley's formula is non-trivial in $p$-extensions $K_m/K$,
$K_m:= K \Q(p^m) \subset \wh \Q$, as soon as $p$ splits in part
in $K$, and gives rare counterexamples (see Section \ref{genus}).

\subsection{Galois action -- Relative submodules} \label{Gmodules}\label{algebra}
Let $p$ be a fixed prime. 
We recall some elementary principles for cyclic extensions
of degree prime to $p$, to apply them to the fields $\Q(N)$, $p \nmid N$.

\smallskip
Let $K=\Q(N)$ and $k=\Q(N')$, $N' \vert N$. The transfer maps 
${\mathcal T}_k \to {\mathcal T}_K$, ${\mathcal R}_k \to {\mathcal R}_K$,
${\mathcal C}_k \to {\mathcal C}_K$, 
are injective and the (arithmetic and algebraic) norms
${\mathcal T}_K\! \to {\mathcal T}_k$,
${\mathcal R}_K\!\to {\mathcal R}_k$,
${\mathcal C}_K \!\to {\mathcal C}_k$,
are surjective since $p \nmid N$. 
More generally, let $({\mathcal M}_{\Q(N)})_{N \geq 1}$ be a family of finite 
$\Z_p[G_N]$-modules, where $G_N = \G(\Q(N)/\Q)$,
provided with natural transfer and norm maps having the above properties 
when $N \not\equiv 0 \!\!\pmod p$ and let 
${\mathcal M}_{\Q(N)}^*$ be the kernel of all the norms 
$\nu_{{\Q(N)}/{\Q(N')}}$, $N' \vert N$, $N'\ne N$, so that
${\mathcal M}_{\Q(N)} \simeq 
\big(\sum_{N'} {\mathcal M}_{\Q(N')} \big)\oplus {\mathcal M}_{\Q(N)}^*$.

\smallskip
Since $G_N$ is cyclic of order $N$, 
the rational characters $\chi$ of $K$ are in one-to-one 
correspondence with the fields $k \subseteq K$; we shall denote 
by $\theta \mid \chi$ the irreducible $p$-adic characters;
each $\theta$ is above a character $\psi$ of 
degree $1$ and order a divisor of $N$. 

\smallskip
We have ${\mathcal M}_K = \oplus_{\chi} {\mathcal M}^{\chi}_K 
= \oplus_{\chi} {\mathcal M}^*_{k}
= \oplus_{\chi} \big[\oplus_{\theta \mid \chi} {\mathcal M}_{k}^{\theta} \big]$. 
Then ${\mathcal M}_K^*$ (or any of its component ${\mathcal M}_K^{\theta_N}$
for $\theta_N$ above $\psi_N$ of order $N$)
is a module over $\Z_p[\mu_N]$, hence isomorphic to 
a product of $\Z_p[\mu_N]$-modules of the form
$\Z_p[\mu_N]/{\mathfrak p}_N^e$, for
${\mathfrak p}_N \mid p$ in $\Q_p(\mu_N)$, $e \geq 1$, 
whose $p$-rank is a multiple of the residue degree $\rho_N^{}$ of $p$ in 
$\Q_p(\mu_N)/\Q_p$; thus $\rho_N^{} \to \infty$ as $N \to \infty$, which is 
considered ``incredible'' for arithmetic invariants, as class groups, for
totally real fields.

\smallskip
Indeed, interesting examples occur more easily when $p$ 
totally splits in $\Q(\mu_N)$ (i.e., $p \equiv 1\!\!\! \pmod N$) and
this ``explains'' the result of \cite{Ich} and \cite{IcNa} claiming that 
$\order C_{\Q(\ell^n)}$ is odd in $\Q(\ell^\infty)$ for all $\ell < 500$, that of 
\cite{Ho4,MorOk1,MorOk2} and explicit deep analytic computations in 
\cite{BPR,FK1,FK2,FKM,Ho3,Ho4,Ich,IcNa,Mor1,Mor2,MorOk1,MorOk2,Wa} 
(e.g., Washington's theorem \cite{Wa} claiming that for $\ell$ and $p$ fixed, 
$\order {\mathcal C}_K$ is constant for all $n$ large enough, whence 
${\mathcal C}_K^*=1$ for all $n \gg 0$, then \cite[Theorems 2, 3, 4, Corollary 1]{FKM}); 
mention also the numerous pioneering Horie's papers proving results of the form: 
``let $\ell_0$ be a small prime; then a prime $p$, totally inert in 
some $\Q(\ell_0^{n_0})$, yields ${\mathcal C}_{\Q(\ell_0^n)}=1$ for all $n$''.
In \cite{BPR}, a conjecture (from ``speculative extensions of 
the Cohen--Lenstra--Martinet heuristics'') implies ${\mathcal C}_{\Q(\ell^n)}^* \ne 1$
for finitely many layers (possibly none). 

\smallskip
Concerning the torsion groups ${\mathcal T}_K$, 
$K=\Q(\ell^n)$, we observe that in general the solutions $p$, for 
$\order {\mathcal T}_K^* \equiv 0 \pmod p$, also
fulfill $p \equiv 1 \pmod {\ell^n}$, which is in some sense
a strong form of Washington's result because the reflection theorem
that we shall recall later in Section \ref{rt}, in
$L:=K(\mu_p)$, the $p$-rank of ${\mathcal T}_K^*$
is bounded by that of a suitable component of ${\mathcal C}_L^*$.
Thus Washington's theorem may be true for the torsion groups in $K$.

\smallskip
One can wonder what happens for the normalized
regulators ${\mathcal R}_K$ and
the relative components ${\mathcal R}_K^*$,
due to the specific nature of a regulator as a Frobenius determinant
and regarding the previous observations.
So, recall some algebraic facts about the ${\mathcal R}_K^*$
that we can explain from heuristics and probabilistic studies given 
in \cite[\S\,4.2.2]{Gra4}. For any real Galois extension $K/\Q$, of Galois 
group $G$, the normalized $p$-adic regulator ${\mathcal R}_K$ may be 
defined via the conjugates of the $p$-adic logarithm of a suitable Minkowski 
unit $\eta$ and can be written, regarding $G$, as Frobenius determinant
${\rm R}_p^G(\eta) = \prd_{\theta} {\rm R}_p^{\theta}(\eta)$,
where $\theta$ runs trough the irreducible $p$-adic characters, and
${\rm R}_p^{\theta}(\eta) = \prd_{\psi \mid \theta} {\rm R}_p^{\psi}(\eta)$ 
with absolutely irreducible characters~$\psi$. Then, in a standard point of view,
${\rm Prob\,} \big({\mathcal R}_K^\theta \equiv 0 \!\! \pmod p \big)
= \ffrac{O(1)}{p^{\,\rho\,\delta^2}}$ (loc. cit.), where $\rho$ is still 
the residue degree of $p$ in the field of values of $\psi$
and $\delta \geq 1$ is a suitable multiplicity of the absolutely irreducible
$\theta$-representation (in our case, $\rho=\rho_{\ell^n}^{}$ and $\delta = 1$). 

\smallskip
Contrary to the class group of $K$ (for $K$ fixed) 
which is {\it finite}, the primes $p$ such that 
${\mathcal R}_K \equiv 0 \! \pmod p$ may be, a priori, infinite 
in number (we have conjectured that it is not the case, but this is 
an out of reach conjecture). Nevertheless, some very large  
$p$ with $\rho=1$, may divide $\order {\mathcal R}_K^{\theta}$, which 
indicates other probabilities conjectured in \cite[Th\'eor\`eme 1.1]{Gra4}.
This analysis also confirms that, for $\ell$ and $p$ fixed, 
${\mathcal T}_{\Q(\ell^n)}$ may be constant for all $n$ large enough.

\smallskip
We have computed the order $\order \wt {\mathcal C}_K$ of the logarithmic class groups
(from \cite{BJ}), and we have no non-trivial example; this means that the 
logarithmic class group behaves, in some sense, as the ordinary $p$-class 
group in $\Q(\ell^\infty)$ or in any $K=\Q(N)$, but not as ${\mathcal T}_K$, 
as we have seen. This is not too surprising since if ${\mathcal C}_K=1$
and if $p$ is totally inert in $K$, then  $\wt {\mathcal C}_K=1$ 
(see \cite[Sch\'ema \S\,2.3]{Jau2} or \cite[Diagram 4.2]{Gra10}).

\subsection{Computation of the structure of ${\mathcal T}_K$ for $K=\Q(\ell^n)$}
\label{nontrivial} \label{T1}
The following PARI/GP programs give the structure, of abelian group, 
of ${\mathcal T}_{\Q(\ell^n)}$, from the polynomial ${\sf P}$
defining $\Q(\ell^n)$:
${\sf P= polsubcyclo(el^{n+1},el^n)}$ for $p>2$
and: ${\sf P=x;for(j=1,n,P=P^2-2)}$, for $p=2$.
These programs are the simplified form of the following general 
one written in \cite[Program I, \S\,3.2]{Gra8}, for 
{\it any monic irreducible polynomial in $\Z[x]$}, where
${\sf r}$ is the number of independent $\Z_p$-extensions:\par
\footnotesize
\begin{verbatim}
PROGRAM I. FOR ANY NUMBER FIELD AND ANY p
{P=x^3-7*x+1;K=bnfinit(P,1);b=2;B=10^5;r=K.sign[2]+1;forprime(p=b,B,Ex=12;
KpEx=bnrinit(K,p^Ex);HpEx=KpEx.cyc;L=List;e=matsize(HpEx)[2];R=0;
for(k=1,e-r,c=HpEx[e-k+1];w=valuation(c,p);if(w>0,R=R+1;
listinsert(L,p^w,1)));if(R>0,print("p=",p," rk(T_p)=",R," T_p=",L)))}
p=7   rk(T_p)=1  T_p=List([7])       p=701 rk(T_p)=1  T_p=List([701])
\end{verbatim}
\normalsize

The parameter ${\sf Ex}$ must be such that ${\sf p^{Ex}}$ is larger 
than the exponent of ${\mathcal T}_K$;
taking ${\sf Ex=2}$ for $p>2$ (resp. ${\sf Ex=3}$ for $p=2$) 
gives the $p$-rank of ${\mathcal T}_K$.\par
\footnotesize
\begin{verbatim}
PROGRAM II. STRUCTURE OF T_K, K=Q(el^n), FOR ANY el, n, p<Bp
{el=2;n=3;Bp=2*10^5;if(el==2,P=x;for(j=1,n,P=P^2-2));if(el!=2,
P=polsubcyclo(el^(n+1),el^n));Ex=6;K=bnfinit(P,1);forprime(p=2,Bp,
KpEx=bnrinit(K,p^Ex);HpEx=KpEx.cyc;L=List;e=matsize(HpEx)[2];R=0;
for(k=1,e-1,c=HpEx[e-k+1];w=valuation(c,p);
if(w>0,R=R+1;listinsert(L,p^w,1)));
if(R>0,print("el=",el," n=",n," p=",p," rk(T)=",R," T=",L)))}

el=2 n=1 p=13  rk(T)=1 T=[13]          el=2 n=3 p=29  rk(T)=1 T=[29] 
el=2 n=1 p=31  rk(T)=1 T=[31]          el=2 n=3 p=521 rk(T)=1 T=[521]
el=2 n=2 p=13  rk(T)=2 T=[169,13]      el=3 n=1 p=7   rk(T)=1 T=[7] 
el=2 n=2 p=31  rk(T)=1 T=[31]          el=3 n=1 p=73  rk(T)=1 T=[73]
el=2 n=2 p=29  rk(T)=1 T=[29]          el=3 n=2 p=7   rk(T)=1 T=[7]         
el=2 n=2 p=37  rk(T)=1 T=[37]          el=3 n=2 p=73  rk(T)=1 T=[73]
el=2 n=3 p=3   rk(T)=2 T=[3,3]         el=5 n=1 p=11  rk(T)=2 T=[11,11]
el=2 n=3 p=31  rk(T)=1 T=[31]          el=5 n=2 p=11  rk(T)=2 T=[11,11]    
el=2 n=3 p=13  rk(T)=2 T=[169,13]      el=5 n=2 p=101 rk(T)=1 T=[101]
el=2 n=3 p=37  rk(T)=1 T=[37]
\end{verbatim}
\normalsize

\begin{remark}\label{T13}
These partial results show that $p$-ramification aspects are 
more intricate since, for instance for $\ell=2$, the divisibility
by $p=29$ only appears for $n=2$ and, for $p = 13$, the $13$-rank and
the exponent increase from $n=1$ to $n=2$.

Unfortunately, it is not possible in practice to compute beyond 
$\ell=17$ with the instruction ${\sf bnfinit}$. 
So, as we have explained in the Introduction, we shall give Section 
\ref{annihilation} another method to test ${\mathcal T}_{\Q(N)} \ne 1$ 
for larger $N$ and $p$; this algorithm only uses elementary basic 
instructions and does not need any large computer memory contrary to 
${\sf bnfinit}$ (see Table \ref{table}). 
\end{remark}

\section{Definition of $p$-adic measures} \label{annihilation}

We recall the main classical principles to apply them to the fields $\Q(N)$, 
with any prime $p \geq 2$, $p \nmid N$.

\subsection{General definition of the Stickelberger elements}
Let $f  > 1$ be any abelian conductor and let $\Q(\mu_f)$ 
be the corresponding cyclotomic field.

We define ${\mathcal S}_{\Q(\mu_f)} :=-\sm_{a=1}^{f} 
\Big(\ffrac{a}{f}-\ffrac{1}{2} \Big) \cdot \Big(\ffrac{\Q(\mu_f)}{a} \Big)^{-1}$
(where the integers $a$ are prime to $f$ and
where Artin symbols are taken over $\Q$).

\smallskip
The properties of annihilation need to multiply ${\mathcal S}_{\Q(\mu_f)}$ 
by a multiplier $1 - c \cdot \big(\frac{\,\Q(\mu_f)}{c} \big)^{-1}$, 
for any odd $c$ prime to $f$; this shall give integral elements in $\Z[\G(\Q(\mu_f)/\Q)]$. 
If $p \cdot f$ is odd, one may take $c=2$ for annihilation in the $\Z_p$-algebra considered
since $\frac{1}{2} \sum_{a=1}^{f} \big(\frac{\Q(\mu_f)}{a} \big)^{-1}$ can be 
neglected at the end.

\smallskip
Put $q=p$ (resp. $4$) if $p \ne 2$ (resp. $p=2$). For $K=\Q(N)$, 
let $L=K(\mu_q)$; to simplify, put 
$K_m:=K \Q(p^m)$, $L_m:= K_m L = K(\mu_{q p^{m}})$ for all $m \geq 0$; 
so $\cup_m K_m = K_\infty$ and $\cup_m L_m = L_\infty$.
All is summarized by the following diagram where 
$G_m \simeq \Z/N \Z \times \Z/p^m\Z \times \Z/\phi(q) \Z$,
$\phi$ being the Euler function:

\unitlength=0.95cm
$$\vbox{\hbox{\hspace{-4.0cm}  
\begin{picture}(11.5,4.25)
\put(6.5,4.0){\line(1,0){3.1}}
\put(6.5,2.50){\line(1,0){3.2}}
\put(7.3,0.50){\line(1,0){2.45}}
\put(3.8,0.50){\line(1,0){1.6}}
\put(4.2,2.50){\line(1,0){1.5}}
\put(6.00,2.8){\line(0,1){0.9}}
\put(6.00,0.8){\line(0,1){1.4}}
\put(10.00,2.8){\line(0,1){0.9}}
\put(10.00,0.8){\line(0,1){1.4}}
\bezier{450}(3.8,0.6)(8.4,1.2)(9.85,2.3)
\put(8.0,1.2){$\sst  G_m$}
\bezier{250}(3.8,0.25)(4.8,0.1)(5.8,0.25)
\put(4.7,-0.05){$\sst G$}
\put(9.7,3.9){$L_\infty\!=\!K_\infty L$}
\put(5.8,3.9){$K_\infty$}
\put(5.8,2.4){$K_m$}
\put(3.0,2.4){$\Q(p^m)$}
\put(3.56,0.8){\line(0,1){1.4}}
\put(9.8,2.4){$L_m\!=\!K_m L$}
\put(9.86,0.4){$L\!=\!K(\mu_q)$}
\put(5.45,0.4){$K\!=\! \Q(N)$}
\put(3.4,0.4){$\Q$}
\put(4.45,0.25){$\sst N$}
\put(3.1,1.45){$\sst p^m$}
\put(7.5,0.25){$\sst \phi(q) = p-1\, or \, \sst 2$}
\end{picture}   }} $$

\subsection{Multipliers of the Stickelberger elements}\label{conductor}
Let $f_N$ be the conductor of $K=\Q(N)$, $N=\ell_1^{n_1}\cdots \ell_t^{n_t}$.
We have $f_N=\ell_1^{n_1+1}\cdots \ell_t^{n_t+1}$ if $N$ is odd
and $f_N=2^{n_1+2}\cdot  \ell_2^{n_2+1}\cdots \ell_t^{n_t+1}$ if $N$ 
is even.
The conductor of $L_m$ is $f_{L_m} = q p^m \cdot f_N$ for $2 \nmid N$ and 
$p^{m+1} \cdot f_N$ otherwise. 

\smallskip
Put $f_{L_m} =: f_N^m$ and let $c$ be prime to $f_N^m$ and, by restriction of
${\mathcal S}_{\Q(\mu_{f_N^m})}$ to $L_m$, let
${\mathcal S}_{L_m}^c :=\ds \Big(1 - c \,\Big(\ffrac{L_m}{c} \Big)^{-1} \Big)\cdot 
{\mathcal S}_{L_m}$;
then ${\mathcal S}_{L_m}^c \in \Z[{G_m}]$. Indeed:
$${\mathcal S}_{L_m}^c=\ffrac{-1}{f_N^m} \sm_a 
\Big[a\, \Big(\ffrac{L_m}{a} \Big)^{-1} - a c \,\Big(\ffrac{L_m}{a} \Big)^{-1}
\Big(\ffrac{L_m}{c} \Big)^{-1}\Big]
+ \ffrac{1-c}{2} \sm_a \Big(\ffrac{L_m}{a} \Big)^{-1}; $$
let $a'_{c} \in [1, f_N^m]$ be the unique integer such that 
$a'_{c} \cdot c \equiv a \pmod {f_N^m}$ and put 
$a'_{c} \cdot  c = a + \lambda^m_a(c) f_N^m$, 
$\lambda^m_a(c) \in \Z$; using the bijection $a \mapsto a'_{c}$ 
in the summation of the second term in $\big [\ \big]$ and 
$\ds\Big(\ffrac{L_m}{a'_{c}} \Big) \Big(\ffrac{L_m}{c} \Big) = \Big(\ffrac{L_m}{a} \Big)$,
this yields:
\begin{eqnarray*} 
{\mathcal S}_{L_m}^c  
&=&\ffrac{-1}{f_N^m}  \Big[ \sm_a a \, \Big(\ffrac{L_m}{a} \Big)^{-1}\!\! - 
\sm_a a'_{c} \cdot c \,\Big(\ffrac{L_m}{a'_{c}} \Big)^{-1} \! \Big(\ffrac{L_m}{c} \Big)^{-1}\Big] 
 + \ffrac{1-c}{2} \sm_a \Big(\ffrac{L_m}{a} \Big)^{-1} \\
& =& \ffrac{-1}{f_N^m} \sm_a \Big [a - a'_{c} \cdot  c \Big] \Big(\ffrac{L_m}{a} \Big)^{-1}
+ \ffrac{1-c}{2}\sm_a \Big(\ffrac{L_m}{a} \Big)^{-1} \\
& =&  \sm_a \Big[\lambda^m_a(c)+ \ffrac{1-c}{2} \Big]
 \Big(\ffrac{L_m}{a} \Big)^{-1} \in \Z[{G_m}].
 \end{eqnarray*}
 
\begin{lemma}\label{sprime} We have the relations
$\lambda^m_{f_N^m-a}(c) + \frac{1-c}{2} = -\big(\lambda^m_{a}(c) + \frac{1-c}{2} \big)$
for all $a \in [1, f_N^m]$ prime to $f_N^m$. Then
${\mathcal S}'^c_{L_m} :=  \sum_{a=1}^{f_N^m/2} \big[\lambda^m_a(c)+ \frac{1-c}{2} \big]
\big(\frac{L_m}{a} \big)^{-1} \in \Z[{G_m}]$ is such that
${\mathcal S}_{L_m}^c = {\mathcal S}'^c_{L_m} \cdot (1-s_\infty)$.
\end{lemma}

\begin{proof} By definition, $(f_N^m-a)'_c$ is in $[1, f_N^m]$ and 
congruent modulo $f_N^m$ to
$(f_N^m-a) \,c^{-1} \equiv - a c^{-1} \equiv - a'_c \pmod {f_N^m}$; thus
$(f_N^m-a)'_c =f_N^m-  a'_c$ and 
$\lambda^m_{f_N^m-a}(c) =\frac{(f_N^m-a)'_c \,c- (f_N^m-a)}{f_N^m}
= \frac{(f_N^m-a'_c)\, c- (f_N^m-a)}{f_N^m} = c-1- \lambda^m_a(c)$, 
whence $\lambda^m_{f_N^m-a}(c) + \frac{1-c}{2} =
 - \big( \lambda^m_{a}(c) + \frac{1-c}{2}  \big)$ and the result.
\end{proof}

\subsection{Spiegel involution}
Let $\kappa_m : {G_m} \to (\Z/q p^m\Z)^\times \simeq 
\G(\Q(\mu_{q p^{m}}^{})/\Q)$
be the {\it cyclotomic character of level $m$},
of kernel $\G(L_m/\Q(\mu_{q p^{m}}^{}))$, defined by
$\zeta^s = \zeta^{\kappa_m(s)}$, for all $s \in {G_m}$ 
and all $\zeta \in \mu_{q p^m}$.
The Spiegel involution is the involution of $(\Z/q p^m \Z) [{G_m}]$ defined by
$x :=\sm_{s \in {G_m}} a_s \cdot s \longmapsto
\ x^* := \sm_{s \in {G_m}} a_s\cdot \kappa_m(s)\cdot s^{-1}$.

Thus, if $s$ is the Artin symbol $\big(\frac{L_m}{a} \big)$, then 
$\big(\frac{L_m}{a} \big)^{\! *} \equiv a\cdot \big(\frac{L_m}{a} \big)^{-1}\!\!
\pmod {qp^m}$.
From Lemma \ref{sprime}, we obtain
${\mathcal S}_{L_m}^{c\,*} = {\mathcal S}'^{c\,*}_{L_m} \cdot (1+s_\infty)$
in $(\Z/q p^m \Z) [{G_m}]$.

\smallskip
We shall use the case $m=0$ for which we have 
$\kappa_m(s) \equiv \omega(s) \pmod q$, where $\omega$ is
the usual Teichm\"uller character $\omega : G_0=\G(L/\Q) \to \Z_p^\times$.

\section{Annihilation theorem of ${\mathcal T}_K^*$} \label{slnc}
Recall that, for $K=\Q(N)$ we put $K_m := K\Q(p^m)$ and $L_m := K_m L$, 
where $L=K(\mu_q)$, $q=p$ or $4$.
For {\it the most precise and straightforward method}, the principle, 
which was given in the 60's and 70's, is to consider the annihilation, by means 
of the above Stickelberger twist, of the kummer radical in $L_m^\times$ 
defining the maximal sub-extension of $H_{K_m}^{\rm pr}$ whose Galois group is of 
exponent $p^m$, then to use the Spiegel involution giving a $p$-adic measure 
annihilating, for $m\to\infty$, the finite Galois group ${\mathcal T}_K$
(see \cite{Gra2, Gra6} for more history).
The case $p=2$ is particularly tricky; to overcome this difficulty, we shall 
refer to \cite{Gra1,Grei}.

\smallskip
In fact, this process is equivalent to get, elementarily, an explicit approximation 
of the $p$-adic $L$-functions ``at $s=1$'', avoiding the ugly computation of Gauss 
sums and $p$-adic logarithms of cyclotomic units \cite[Theorem 5.18]{Wa}.
We have the following result with a detailed proof in \cite[Theorems 5.3, 5.5]{Gra6}:

\begin{proposition} For  $p \geq 2$, let $p^e$ be the exponent of 
${\mathcal T}_K$
for $K=\Q(N)$. For all $m \geq e$, the $(\Z/qp^m \Z)[{G_m}]$-module 
${\mathcal T}_K$ is annihilated by ${\mathcal S}'^{c\,*}_{L_m}$.
\end{proposition}

From the expression of ${\mathcal S}'^c_{L_m}$ (Lemma \ref{sprime}), 
the Spiegel involution yields:
\begin{equation}
{\mathcal S}'^{c\,*}_{L_m} \equiv  
 \sm_{a=1}^{f_N^m/2}  \Big[ \lambda^m_a(c) + \ffrac{1-c}{2} \Big]\, 
 a^{-1} \Big(\ffrac{L_m}{a} \Big) \pmod {q p^m},
\end{equation}
defining a coherent family in $\ds\limproj_{m \geq e} (\Z/qp^m \Z) [{G_m}]$.
One obtains, by restriction of ${\mathcal S}'^{c\,*}_{L_m}$ to $K$, a 
coherent family of annihilators of ${\mathcal T}_K$, whose $p$-adic limit
$$\hbox{${\mathcal A}'^c_K :=
\ds\lim_{m \to \infty} \sm_{a=1}^{f_N^m/2} \Big[ \lambda^m_a(c) 
+ \ffrac{1-c}{2} \Big]\, a^{-1} \Big(\ffrac {K}{a} \Big)$ in $\Z_p[\G(K/\Q)]$,}$$ 
is a canonical annihilator of ${\mathcal T}_K$.

\begin{remark}
Let $\alpha_{L_m}^*:= \big[\sum_{a=1}^{f_N^m}  \big(\frac{L_m}{a} \big)^{-1} \big]^*
\!\equiv \sum_{a=1}^{f_N^m} a^{-1}  \big(\frac{L_m}{a} \big)\!\!\! \pmod{qp^m}$; then:
$\alpha_{L_m}^* := \!\sum_{a=1}^{f_N^m/2} a^{-1}  \big(\frac{L_m}{a} \big) +
(f_N^m-a)^{-1}  \big(\frac{L_m}{f_N^m-a} \big)\! \equiv\!
\sum_{a=1}^{f_N^m/2} a^{-1} \big(\frac{L_m}{a} \big)(1-s_\infty) \,{\rm mod}\, {f_N^m}$,

\smallskip
\noindent
which annihilates ${\mathcal T}_K$, modulo $p$, by restriction since $K$ is real.
We shall neglect $\frac{1-c}{2}\cdot \alpha_{L_m}^*$ and we still denote 
${\mathcal A}'^c_K := \ds\lim_{m \to \infty} \Big[\hbox{ $\sm_{a=1}^{f_N^m/2} 
\lambda^m_a(c) \, a^{-1} \Big(\ffrac{K}{a} \Big)$} \Big]$.  
\end{remark}

\begin{lemma}\label{psi}
For $K=\Q(N)$, $\psi_N$ of order $N$ and conductor $f_N$, one has
$\psi_N({\mathcal A}'^c_K) = (1-\psi_N (c)) \cdot \hbox{$\frac{1}{2}$} L_p(1,\psi_N)$.
\end{lemma}

\begin{proof}
Classical construction of $p$-adic $L$-functions (e.g., \cite[Propositions II.2, II.3, 
D\'efinition II.3, II.4, Remarques II.3, II.4]{Gra2}, after Amice--Fresnel works,
then \cite[Chapters 5, 7]{Wa}). For more details, see \cite[\S\,7.1]{Gra6}.
\end{proof}

\begin{proposition} \label{thmp}
Let $K := \Q(N)$, $N = \ell_1^{n_1} \cdots \ell_t^{n_t}$, and let $p \nmid N$. 
Then, for the $p$-adic character $\theta_N$ 
above a character $\psi_N$ of order $N$ of $K$, the component 
${\mathcal T}_K^{\theta_N}$ is annihilated by $(1 - \psi_N (c)) \cdot 
\hbox{$\frac{1}{2}$}L_p(1,\psi_N)$. 
Moreover, from the principal theorem of Ribet--Mazur--Wiles--Kolyvagin--Greither
on abelian fields, $\hbox{$\frac{1}{2}$}L_p(1,\psi_N)$ gives its order.
\end{proposition}

In the practice, taking $c=2$ in the programs when $N$ is odd and $p \ne 2$, 
we obtain the annihilation by $(1 - \psi_N (2))  \cdot \hbox{$\frac{1}{2}$}L_p(1,\psi_N)$, 
where $\psi_N (2)$ is a root of unity of order dividing $N$, hence prime
to $p$; thus $(1 - \psi_N (2))$ is invertible modulo $p$, except when $\psi_N(2) = 1$ 
for $N = 1093$, $3511$,\,$\ldots$ which are in fact unfeasible numerically. 
If $p=2$ an odd $c$ prime to $N$ must be chosen.

\begin{lemma}\cite[Corollary 7.3\,(iii)]{Gra6}.  \label{norm}
We have ${\mathcal A}'^c_K \equiv \! \sm_{a=1}^{q f_N/2}
\lambda_a^0(c) \, a^{-1} \Big(\ffrac{K}{a} \Big)$ (mod $p$), 
whence $\psi_N({\mathcal A}'^c_K)
\equiv (1 - \psi_N (c))  \cdot \hbox{$\frac{1}{2}$}L_p(1,\psi_N) \pmod p$.
\end{lemma}

Thus, we have obtained a computable 
characterization of non-triviality of ${\mathcal T}_K$, for $K=\Q(N)$,
$N=\ell_1^{n_1} \cdots \ell_t^{n_t}$, and for any $p \geq 2$, $p \nmid N$,
where $f_N$ is the conductor of $K$ (see \S\,\ref{conductor}):

\begin{theorem}\label{st*} Let $L=K(\mu_q)$, $q=p$ or $4$ as usual.
Let $c$ be an integer prime to $q\,f_N$. 
For all $a \in [1,q\,f_N]$, prime to $q\,f_N$, let $a'_c$ be the 
unique integer in $[1, q\,f_N]$ such that $a'_{c} \cdot c \equiv a \pmod {q\,f_N}$ 
and put $a'_c \cdot  c - a = \lambda_a(c)\, q\,f_N$, $\lambda_a(c) \in \Z$.
Let ${\mathcal A}'^c_K \equiv 
\sm_{a=1}^{q\,f_N/2} \lambda_a(c) \, a^{-1} \Big(\ffrac{K}{a} \Big) \pmod p$,
let $\psi_N$ be a character of $K$ of order $N$ and $\theta_N$ the $p$-adic character 
above $\psi_N$.
Then, if $c$ is chosen such that $\psi_N(c) \ne 1$, the $\theta_N$-component of 
the $\Z_p[\G(K/\Q)]$-module ${\mathcal T}_K$ is non-trivial if and only if 
$\psi_N({\mathcal A}'^c_K)$ is not a $p$-adic unit.
\end{theorem}

\subsection{Numerical test ${\mathcal T}^*_{\Q(\ell^n)} \ne 1$ 
for $\ell>2$, $p>2$} \label{programIII-IV}

We have, from \S\,\ref{algebra}, ${\mathcal T}_{\Q(\ell^n)} = 
{\mathcal T}_{\Q(\ell^n)}^* \plus {\mathcal T}_{\Q(\ell^{n-1})}$. 
For a character $\psi$ of order $\ell^n$ of $K$,
the condition $\psi({\mathcal A}'^c_{\Q(\ell^n)}) \equiv 0 \pmod {{\mathfrak p}}$, for 
some ${\mathfrak p} \mid p$, is equivalent to the non-triviality of 
${\mathcal T}_{\Q(\ell^n)}^*$, due to the $p$-adic character $\theta$ above $\psi$.
We compute $\psi({\mathcal A}'^c_{\Q(\ell^n)}) \!\!\pmod p$ and test if
the norm of this element is divisible by $p$; this characterize the
condition ${\mathcal T}^*_{\Q(\ell^n)} \ne 1$:
\footnotesize
\begin{verbatim} 
PROGRAM III. TEST #T*>1 WITH NORM COMPUTATIONS FOR el>2, p>2
{C=2;forprime(el=3,120,for(n=1,4,Q=polcyclo(el^n);
h=znprimroot(el^(n+1));H=lift(h);forprime(p=3,2500,if(p==el,next);
f=p*el^(n+1);cm=Mod(C,f)^-1;g=znprimroot(p);G=lift(g);gm=g^-1;
e=lift(Mod((1-H)*el^(-n-1),p));H=H+e*el^(n+1);h=Mod(H,f);
e=lift(Mod((1-G)*p^-1,el^(n+1)));G=G+e*p;g=Mod(G,f);
S=0;hh=1;gg=1;ggm=1;for(u=1,el^n*(el-1),hh=hh*h;
t=0;for(v=1,p-1,gg=gg*g;ggm=ggm*gm;a=lift(hh*gg);A=lift(a*cm);
t=t+(A*C-a)/f*ggm);S=S+lift(t)*x^u);s=Mod(S,Q);vp=valuation(norm(s),p);
if(vp>0,print("el=",el," n=",n," p=",p)))))}
\end{verbatim}
\normalsize

The program finds again very quickly the cases of Table \ref{T1}:
($\ell^n=3$, $p = 7$, $p=73$), 
($\ell^n=27$, $p=109$), ($\ell^n=81$, $p=487$, $p=1621$), etc.,
($\ell=5$, $p=11$), ($\ell=25$, $p=101$, $p=1151$, $p=2251$), 
($\ell=125$, $p=2251$), etc. 
 
\smallskip
Interesting cases are $\ell = 5$ giving
${\mathcal T}_{\Q(5^2)} \simeq \Z/2251\Z$ and 
${\mathcal T}^*_{\Q(5^3)} \simeq \Z/2251\Z$; which implies that
${\mathcal T}_{\Q(5^3)}$ contains $(\Z/2251\Z)^2$.

\smallskip
To verify, we have computed, with Program II of \S\,\ref{T1}, 
the structure of ${\mathcal T}_{\Q(\ell^n)}$
for $\ell^n=27$, $p=109$, which is much longer and needs an huge computer 
memory; we get as expected ${\sf el=3\  n=3\  p=109\ \  rk(T)=1\ \ T=[109]}$.
 
Whence, we can propose the following program, only
considering primes $p \equiv 1 \!\!\pmod{\ell^n}$, so that
$p$ splits completely in $\Q(\mu_{\ell^n})$ which allows
to characterize, once for all, a prime ${\mathfrak p} \mid p$ by
means of a congruence ${\sf z \equiv r \ (mod \,{\mathfrak p}})$,
where ${\sf z}$ denotes, in the program, a generator of $\mu_{\ell^n}$
and ${\sf r}$ a rational integer, then avoiding the computation 
of ${\sf N=norm(s)}$ in some programs, which takes too much time. 
We then find supplementary examples.\par
\footnotesize
\begin{verbatim}
PROGRAM IV. TEST #T*>1 MODULO (zeta-r) WHEN p=1 (mod el^n) FOR el>2, p>2
{C=2;forprime(el=3,250,for(n=1,6,Q=polcyclo(el^n);h=znprimroot(el^(n+1));
H=lift(h);forprime(p=3,5000,if(Mod(p,el^n)!=1,next);Qp=Mod(1,p)*Q;
m=(p-1)/el^n;r=znprimroot(p)^m;f=p*el^(n+1);cm=Mod(C,f)^-1;
g=znprimroot(p);G=lift(g);gm=g^-1;
e=lift(Mod((1-H)*el^(-n-1),p));H=H+e*el^(n+1);h=Mod(H,f);
e=lift(Mod((1-G)*p^-1,el^(n+1)));G=G+e*p;g=Mod(G,f);
S=0;hh=1;gg=1;ggm=1;for(u=1,el^n*(el-1)/2,hh=hh*h;
t=0;for(v=1,p-1,gg=gg*g;ggm=ggm*gm;a=lift(hh*gg);A=lift(a*cm);
t=t+(A*C-a)/f*ggm);S=S+lift(t)*x^u);s=lift(Mod(S,Qp));
R=1;for(k=1,el^n,R=R*r;if(Mod(k,el)==0,next);t=Mod(s,x-R);
if(t==0,print("el=",el," n=",n," p=",p))))))}

PROGRAM V. VARIANT FOR ANY NUMBER d OF p-PLACES 
USING THE FACTORIZATION OF Q mod p
d (a power of el) may be optionally specified (e.g. d=1,el,...):

{el=3;C=2;for(n=1,10,Q=polcyclo(el^n);h=znprimroot(el^(n+1));H=lift(h);
forprime(p=5,2*10^4,f=p*el^(n+1);cm=Mod(C,f)^-1;Qp=Mod(1,p)*Q;
F=factor(Q+O(p));R=lift(component(F,1));d=matsize(F)[1];
g=znprimroot(p);G=lift(g);gm=g^-1;
e=lift(Mod((1-H)*el^(-n-2),p));H=H+e*el^(n+2);h=Mod(H,f);
e=lift(Mod((1-G)*p^-1,el^(n+2)));G=G+e*p;g=Mod(G,f);
S=0;hh=1;gg=1;ggm=1;for(u=1,el^n*(el-1)/2,hh=hh*h;
t=0;for(v=1,p-1,gg=gg*g;ggm=ggm*gm;a=lift(hh*gg);A=lift(a*cm);
t=t+(A*C-a)/f*ggm);S=S+lift(t)*x^u);s=lift(Mod(S,Qp));
for(k=1,d,t=Mod(s,R[k]);if(t==0,print("el=",el," n=",n," p=",p)))))}
\end{verbatim}
\normalsize

The following table is the addition of that obtained with Programs III--V:
\footnotesize
\begin{verbatim}
el=3   n=1  p=7         el=5    n=2  p=6701      el=67   n=1  p=269
el=3   n=1  p=73        el=5    n=3  p=2251      el=83   n=1  p=499
el=3   n=3  p=109       el=5    n=3  p=27751     el=101  n=1  p=607
el=3   n=3  p=17713     el=5    n=4  p=11251     el=107  n=1  p=857
el=3   n=4  p=487       el=17   n=1  p=239       el=109  n=1  p=50359
el=3   n=4  p=1621      el=23   n=1  p=47        el=131  n=1  p=2621
el=3   n=7  p=17497     el=29   n=1  p=59        el=131  n=1  p=8123
el=5   n=1  p=11        el=37   n=1  p=4441      el=131  n=1  p=34061
el=5   n=2  p=101       el=43   n=1  p=173       el=137  n=1  p=1097
el=5   n=2  p=1151      el=47   n=1  p=283       el=151  n=1  p=907
el=5   n=2  p=2251      el=61   n=1  p=1709      el=191  n=1  p=383
\end{verbatim}
\normalsize

In the case $p=2$, $\ell>2$, we have the exceptional prime $\ell = 11$
for which $3$ splits in $\Q(11)$, whence $1 - \psi_{11}(3)=0$, giving a 
wrong solution with $C=3$. Moreover, the $2$-adic characters 
$\theta$ of $\Q(\ell^n)$ cannot be of degree $1$ in practice since $2$ is inert
in $\Q(\ell)$ except for the two known cases of non-trivial Fermat
quotients of $2$ modulo $\ell$; so we are obliged to test with the 
computation of a norm in $\Q(\mu_{\ell})$. 
As expected, any program gives the solutions $\ell=1093$, $n=1$, $p=2$ 
and $\ell=3511$, $n=1$, $p=2$. For $\ell = 1093$, see Remarks \ref{rema1093}.

\subsection{Numerical test ${\mathcal T}^*_{\Q(2^n)} \ne 1$ 
for $\ell = 2$, $p>2$}\label{programVI-VIII}

We have only to modify the conductor $p\,2^{n+2}$ of $L=K(\mu_p)$
where $K=\Q(2^n)$, then note that we must choose another multiplier 
for the Stickelberger element and the generator ${\sf h=Mod(5,el^{(n+2)})}$
(for $p=3$ one must take $C=5$ giving the solution ${\sf el=2\ \ n=3\ \ p=3}$).
To obtain a half-system $S$ for $a \in [1, p\,2^{n+2}]$ we can neglect the subgroup 
generated by the generator of $\G(\Q(\mu_f)/\Q(2^n)(\mu_p))$
(will be proven in the general case in Lemma \ref{half}):\par
\footnotesize
\begin{verbatim}
PROGRAM VI. TEST #T>1 WITH NORM COMPUTATIONS FOR el=2, p>3
{el=2;for(n=1,8,Q=polcyclo(el^n);h=Mod(5,el^(n+2));H=lift(h);C=3;
forprime(p=5,2*10^4,f=p*el^(n+2);cm=Mod(C,f)^-1;
g=znprimroot(p);G=lift(g);gm=g^-1;
e=lift(Mod((1-H)*el^(-n-2),p));H=H+e*el^(n+2);h=Mod(H,f);
e=lift(Mod((1-G)*p^-1,el^(n+2)));G=G+e*p;g=Mod(G,f);
S=0;hh=1;gg=1;ggm=1;for(u=1,el^n,hh=hh*h;
t=0;for(v=1,p-1,gg=gg*g;ggm=ggm*gm;a=lift(hh*gg);A=lift(a*cm);
t=t+(A*C-a)/f*ggm);S=S+lift(t)*x^u);s=Mod(S,Q);
vp=valuation(norm(s),p);if(vp>0,print("el=",el," n=",n," p=",p))))}
\end{verbatim}
\normalsize

For instance the results ${\sf el=2}$ ${\sf n=1}$ ${\sf p=13}$, 
${\sf el=2}$ ${\sf n=2}$ ${\sf p=13}$
correspond to the following cases of Table \ref{T1}:
${\sf el^n=2}$  ${\sf p=13}$ ${\sf rk(T)=1}$ ${\sf T=[13]}$ and
${\sf el^n=4}$  ${\sf p=13}$ ${\sf  rk(T)=2}$ ${\sf T=[169,13]}$.

\smallskip
As for $\ell >2$, we have a faster program using only primes
$p \equiv 1\!\! \pmod{2^n}$, which gives new solutions (e.g., $\ell^n=2^{10}$, 
$p=114689$). The table below is the addition of that obtained with these
two programs:\par
\footnotesize
\begin{verbatim}
PROGRAM VII. TEST #T*>1 MODULO (zeta-r) WHEN p=1 (mod el^n) FOR el=2, p>3
{el=2;for(n=1,12,Q=polcyclo(el^n);h=Mod(5,el^(n+2));H=lift(h);C=3;
forprime(p=5,2*10^5,if(Mod(p,el^n)!=1,next);f=p*el^(n+2);cm=Mod(C,f)^-1;
Qp=Mod(1,p)*Q;m=(p-1)/el^n;r=znprimroot(p)^m;
g=znprimroot(p);G=lift(g);gm=g^-1;
e=lift(Mod((1-H)*el^(-n-2),p));H=H+e*el^(n+2);h=Mod(H,f);
e=lift(Mod((1-G)*p^-1,el^(n+2)));G=G+e*p;g=Mod(G,f);
S=0;hh=1;gg=1;ggm=1;for(u=1,el^n,hh=hh*h;
T=0;for(v=1,p-1,gg=gg*g;ggm=ggm*gm;a=lift(hh*gg);A=lift(a*cm);
T=T+(A*C-a)/f*ggm);S=S+lift(T)*x^u);s=lift(Mod(S,Qp));
R=1;for(k=1,el^n,R=R*r;if(Mod(k,el)==0,next);t=Mod(s,x-R);
if(t==0,print("el=",el," n=",n," p=",p)))))}

el=2   n=1   p=13      el=2   n=3   p=3         el=2   n=7   p=257      
el=2   n=1   p=31      el=2   n=3   p=521       el=2   n=7   p=641
el=2   n=2   p=13      el=2   n=5   p=3617      el=2   n=8   p=18433          
el=2   n=2   p=29      el=2   n=5   p=4513      el=2   n=10  p=114689
el=2   n=2   p=37      el=2   n=6   p=193      

PROGRAM VIII. VARIANT USING THE FACTORIZATION OF Q (mod p)
for any number d of p-places of K
{el=2;for(n=1,12,Q=polcyclo(el^n);h=Mod(5,el^(n+2));H=lift(h);C=3;
forprime(p=5,2*10^5,f=p*el^(n+2);cm=Mod(C,f)^-1;Qp=Mod(1,p)*Q;
F=factor(Q+O(p));R=lift(component(F,1));d=matsize(F)[1];
\\d (a power of 2) may be optionally specified (e.g. d=1, d=2)
g=znprimroot(p);G=lift(g);gm=g^-1;
e=lift(Mod((1-H)*el^(-n-2),p));H=H+e*el^(n+2);h=Mod(H,f);
e=lift(Mod((1-G)*p^-1,el^(n+2)));G=G+e*p;g=Mod(G,f);
S=0;hh=1;gg=1;ggm=1;for(u=1,el^n,hh=hh*h;
T=0;for(v=1,p-1,gg=gg*g;ggm=ggm*gm;a=lift(hh*gg);A=lift(a*cm);
T=T+(A*C-a)/f*ggm);S=S+lift(T)*x^u);s=lift(Mod(S,Qp));
for(k=1,d,t=Mod(s,R[k]);if(t==0,print("el=",el," n=",n," p=",p)))))}
\end{verbatim}
\normalsize

\subsection{Test on the normalized $p$-adic regulator}\label{programR}
A {\it sufficient condition} to get the divisibility of $\order {\mathcal C}_K$ 
by $p$, when we have obtained ${\mathcal T}_K \ne 1$, is to establish 
that the normalized $p$-adic regulator ${\mathcal R}_K$ is a 
$p$-adic unit; otherwise, this only gives that very 
probably $\order {\mathcal C}_K = 1$. 
Since with PARI/GP the computation of units implies that of the
class number (due to ${\sf K=bnfinit(P)}$), there is no interest 
to test the $p$-divisibility of the regulator instead of looking at ${\sf K.no}$
(the class number), except to obtain some verifications and to get
the $p$-adic relations between the units.
 
\smallskip
The following program computes the $p$-rank of the matrix $M$ 
obtained by approximation (modulo~$p$) of the $p$-adic expressions 
$\frac{1}{p}{\rm log}_p(\varepsilon_i)$, written on a 
$\Z$-basis of $K$ (via ${\sf nfalgtobasis}$), for a system of fundamental units 
$\varepsilon_i$ given by PARI/GP; then ${\mathcal R}_K$ is a $p$-adic unit if and 
only if ${\rm rank}(M)=\ell^n-1$:\par
\footnotesize
\begin{verbatim}
PROGRAM IX. TEST ON THE REGULATOR R FOR el>2, n>=1
{el=17;n=1;p=239;N=el^n;if(el==2,P=x;for(j=1,n,P=P^2-2));
if(el!=2,P=polsubcyclo(el^(n+1),el^n));Pp=P*Mod(1,p^2);
K=bnfinit(P,1);E=K.fu;L=List;for(k=1,N-1,e=E[k];
ep=Mod(lift(e),Pp);epm=Mod(lift(e^-1),Pp);
Ep=ep;for(u=1,N,Ep=Ep^p);Ep=Ep*epm;le=lift(Ep-1);Le=0;
for(j=0,N,c=polcoeff(le,j);c=lift(c);Le=Le+c*x^j);
Le=nfalgtobasis(K,Le)/p;LogE=Mod(Le,p);listinsert(L,LogE,1));
M=matrix(N-1,N,i,j,polcoeff(L[i],j));rk=matrank(M);
if(rk==N-1,print("N =",N," p=",p," rk(M)=",rk," R_K invertible"));
if(rk<N-1,print("N =",N," p=",p," rk(M)=",rk," R_K=0 mod (p)"))}

N=3 p=7  rk(M)=1 R_K=0 mod (p)    N=17 p=239 rk(M)=15 R_K=0 mod (p)
N=3 p=73 rk(M)=1 R_K=0 mod (p)    N=23 p=47  rk(M)=21 R_K=0 mod (p)
N=5 p=11 rk(M)=2 R_K=0 mod (p)    N=29 p=59  rk(M)=27 R_K=0 mod (p)

N=2 p=13 rk(M)=0 R_K=0 mod (p)    N=4 p=29   rk(M)=2  R_K=0 mod (p)
N=2 p=31 rk(M)=0 R_K=0 mod (p)    N=4 p=37   rk(M)=2  R_K=0 mod (p)
N=4 p=13 rk(M)=1 R_K=0 mod (p)    N=8 p=521  rk(M)=6  R_K=0 mod (p)
\end{verbatim}
\normalsize

\subsection{Conjecture about the $p$-torsion groups ${\mathcal T}_{\Q(N)}$}
The annihilation Theorem \ref{st*} allows us to test the non-triviality 
of the integer $\order {\mathcal T}_K = \order {\mathcal C}_K \cdot 
\order {\mathcal R}_K \cdot \order {\mathcal W}_K$, for $K:= {\Q(N)}$, 
giving possible non-trivial class groups (see Lemma \ref{lemmaW} about 
${\mathcal W}_K$, in general trivial). More precisely, all computations 
or experiments depend on the relative components ${\mathcal T}_K^*$ 
whose orders are given by $\frac{1}{2} \, L_p(1,\psi_N)$, for $\psi_N$
of order $N$ of $K$.

\smallskip
Indeed, we do not see why $\order {\mathcal C}_K$ should be 
always trivial for an``algebraic reason'', even if it is known that 
${\mathcal R}_K$ may be, a priori, non-trivial
whatever the order of magnitude of $p$. Moreover,
an observation made in other contexts shows that, when 
$\order {\mathcal C}_K^* \cdot 
\order {\mathcal R}_K^*$ is non-trivial, the probability
of $\order {\mathcal R}_K^* \ne 1$ is, roughly, $p$ times that of
$\order {\mathcal C}_K^* \ne 1$.
The Cohen--Lenstra--Martinet heuristics (see \cite{BPR,Mi1,Mi2} for large
developments) give low probabilities for non-trivial $p$-class groups,
even in the case of residue degree $1$ of $p$ in $\Q(\mu_N)$.

\smallskip
As for the question of $p$-rationality of number fields, when $K \subset \wh \Q$ 
is fixed, the number of $p$ such that $\order {\mathcal T}_K^* 
\equiv 0 \pmod p$ may be finite as we have conjectured; whence the rarity 
of these cases.
Nevertheless, we propose the following conjecture claiming the infiniteness 
of non-trivial relative groups ${\mathcal T}_K^*$ 
when all parameters vary (i.e., the infiniteness of Table \ref{table}):
\begin{conjecture} 
There exist infinitely many pairs $(N, p)$, $N \geq 2$, $p$ prime, $p \nmid N$, 
such that $\frac{1}{2} \, L_p(1,\psi_N) \equiv 0 \pmod {{\mathfrak p}_N}$, for 
some ${\mathfrak p}_N \mid p$ in $\Q(\mu_N)$,
where $\psi_N$ is a character of $\Q(\mu_N)$ of order $N$ (whence 
${\mathcal T}_{\Q(N)}^* \ne 1$).
\end{conjecture}

We have seen that the solutions $p$ to ${\mathcal T}_K^* \ne 1$, in the case
$K=\Q(\ell^n)$, are mostly of the form $p = 1 + \lambda\,\ell^n$ giving, possibly, a 
class group of $K$ roughly of order $O(\ell^n)$, which is very reasonable 
since the discriminant of $K$ is such that $\sqrt {D_K} = (\ell^n)^{O(\ell^n)}$,
whereas the class number fulfills the following general property 
$\order  C_K \leq c_{\ell^n,\epsilon} \! \cdot (\sqrt {D_K})^{1+\epsilon}$ 
\cite{ABC} and the $\epsilon$-conjecture 
$\order  C_K \leq c'_{\ell^n,\epsilon} \! \cdot (\sqrt {D_K})^{\epsilon}$.

\smallskip
Finally, if we assume that the $p$-class group ${\mathcal C}_K$
and the regulator ${\mathcal R}_K$ are random and
independent, the Weber class number 
conjecture is possibly false for some $\ell_0^{n_0}$, $p_0$,
the prime $\ell = 2$ being not specific.

\section{Reflection theorem  for $p$-class groups and $p$-torsion groups}\label{rt}
Reflection theorem compares the $p$-class group ${\mathcal C}_K$ of $K$ 
with a suitable component of the $p$-torsion group ${\mathcal T}_L$ of $L:=K(\mu_p)$.

\smallskip
Put ${\rm rk}_p(A) := {\rm dim}_{\F_p}(A/A^p)$ for any 
abelian group $A$ of finite type.

\subsection{Case $p = 2$}
Consider, once for all, the case $p=2$ with $2 \nmid N$. The reflection
theorem works in $K = \Q(N)$, with the trivial character; 
applied with the set $S$ of prime ideals of $K$ above $2$,
it is given by \cite[Proposition III.4.2.2, \S\,II.5.4.9.2]{Gra0}, where 
${\mathfrak m}^* = (4)$ and where ${\mathcal C}_K^{(4)}$ denotes
a ray class group modulo $(4)$. We have, in reflection theorems, the relation
${\mathcal T}_K^{\rm res} \simeq {\mathcal T}_K^{\rm ord} \plus \F_2^N$ 
\cite[Theorem III.4.1.5]{Gra0}, valid under Leopoldt's conjecture for $p=2$. 

\begin{theorem}
We have, in $K=\Q(N)$, for any odd $N > 1$ and $p=2$:
\begin{eqnarray*}
{\rm rk}_2 ({\mathcal T}_K^{\rm ord}) & = &
{\rm rk}_2 \big[{\mathcal C}_K^{\rm res}/cl_K^{\rm res} (S) \big] 
+ \order S - 1, \\ 
{\rm rk}_2 ({\mathcal T}_K^{\rm ord}) & = &
{\rm rk}_2 \big[{\mathcal C}_K^{\rm ord}/cl_K^{\rm ord} (S) \big] 
+ \order S -1, \\
{\rm rk}_2 ({\mathcal C}_K^{(4) \, {\rm ord}}) & = &
{\rm rk}_2 ({\mathcal C}_K^{\rm res}),  
\end{eqnarray*}
Thus, ${\mathcal T}_K^{\rm ord} =1$ 
(i.e., ${\mathcal C}_K^{\rm ord} = {\mathcal R}_K^{\rm ord} 
= {\mathcal W}_K^{\rm ord} =1$)
if and only if $2$ is inert in $K/\Q$ and ${\mathcal C}_K^{\rm ord} = 1$
(or $2$ is inert and ${\mathcal C}_K^{\rm res} = 1$). 
\end{theorem}

\begin{remark} \label{rema1093}
Let $K=\Q(N)$, $N$ odd. If $p=2$ is inert in $K$,
${\rm rk}_2({\mathcal T}_K^{\rm ord})=$ ${\rm rk}_2( {\mathcal C}_K^{\rm res}) 
= {\rm rk}_2( {\mathcal C}_K^{\rm ord})$ (second and third formulas). 
This does not apply if $N$ is divisible by $\ell = 1093$, $3511$ 
and primes $\ell$ such that $2^{\ell-1} \equiv 1 \pmod {\ell^2}$.
For $\ell = 1093$ and from
${\rm rk}_2({\mathcal T}_K^{\rm ord}) = 
{\rm rk}_2({\mathcal C}_K^{\rm ord}/cl_K^{\rm ord} (S) ) + 1092$, 
we have verified that the norm of 
$(1-\psi(3)) \cdot \hbox{$\frac{1}{2}$}L_p(1,\psi)$
is exactly $2^{1092}$; this means that $2$ annihilates 
${\mathcal T}_K^{\rm ord}$, whence that ${\mathcal C}_K^{S\, \rm ord} =1$ 
and that ${\mathcal T}_K^{\rm ord} \simeq (\Z/2\Z)^{1092}$.
\end{remark}

\subsection{Case $p \ne 2$}
The application of the reflection theorem needs to consider
$L=K \Q(\mu_p)$ for $K=\Q(N)$, $p \nmid N$, with the group $\G(L/K)$.  
Let $\omega_p =: \omega$ be the Teichm\"uller character. 
We denote by ${\rm rk}_\chi(A)$ the $\F_p$-dimension of the 
$\chi$-component of $A/A^p$, $\chi \in \langle \omega \rangle$; 
whence ${\rm rk}_1(A) = {\rm rk}_p(A)$. 

Since $p$ is totally ramified 
in $L/K$ one may denote $S$, by abuse, the sets of $p$-places 
of $K$ and $L$, respectively; we then have $cl_L(S) \simeq cl_K(S)$. 

\begin{theorem}\cite[\S\,II.5.4.2 and Theorem II.5.4.5]{Gra0} \label{ranks}
Let $p>2$ be a prime not dividing $N$ and let $K := \Q(N)$, $L := K(\mu_p)$; 
put ${\mathfrak P}^* = (p)\! \cdot\! (1-\zeta_p)$ in~$L$, 
where ${\mathcal C}_L^{{\mathfrak P}^*}$ is
the ray class group of modulus ${\mathfrak P}^*$. We have:
\begin{eqnarray*}
{\rm rk}_p ({\mathcal T}_K) & = &
{\rm rk}_\omega ({\mathcal C}_L), \\ 
{\rm rk}_p \big[ {\mathcal C}_K/cl_K (S_K)  \big] & = &
{\rm rk}_\omega ({\mathcal T}_L^{}) +1 - \order S_K, \\
{\rm rk}_p ({\mathcal C}_K) & = &
{\rm rk}_\omega ({\mathcal C}_L^{{\mathfrak P}^*}) + 1 - N,  
\end{eqnarray*}
\end{theorem}

\subsection{Illustration of reflection theorem for $N=\ell^n$, $p \ne 2$}
The parameter ${\sf \order zp}$ gives the number $\ell^n\,(p-1)/2+1$
of $\Z_p$-extensions of $L$, but the cyclotomic extension of $\Q$
does not intervene because the conductor of $\Q(p)$ is $p^2$ larger 
that ${\mathfrak P}^*$; thus, ${\sf \order zp -1 - rk(Hp)}$ is the $p$-rank of the 
torsion part, where ${\sf Hp}$ is the ray class group ${\mathcal C}_L^{{\mathfrak P}^*}$
(e.g., $\ell=2$, $p=11, 13, 19$).\par
\footnotesize
\begin{verbatim}
PROGRAM X. ILLUSTRATION OF FORMULA (8)
{el=2;for(n=1,3,print("el=",el," n=",n);if(el==2,P=x;for(j=1,n,P=P^2-2));
if(el!=2,P=polsubcyclo(el^(n+1),el^n));forprime(p=2,23,if(p==el,next);
Q=polcompositum(P,polcyclo(p))[1];L=bnfinit(Q,1);
r=el^n*(p-1)/2+1;A=idealfactor(L,p);d=matsize(A)[1];a=1;
for(k=1,d,a=idealmul(L,a,component(A,1)[k]));ap=idealpow(L,a,p);
Lp=bnrinit(L,ap);Hp=Lp.cyc;LT=List;e=matsize(Hp)[2];
R=0;for(k=1,e,c=Hp[e-k+1];w=valuation(c,p);if(w>0,R=R+1;
listinsert(LT,p^w,1)));print("p=",p," rk(Hp)=",R," #zp=",r," Hp=",LT)))}
el=2  n=1
p=3  rk(Hp)=2  #zp=3  Hp=[3,3]
p=5  rk(Hp)=4  #zp=5  Hp=[5,5,5,5]
p=7  rk(Hp)=6  #zp=7  Hp=[7,7,7,7,7,7]
p=11 rk(Hp)=11 #zp=11 Hp=[121,11,11,11,11,11,11,11,11,11,11]
p=13 rk(Hp)=13 #zp=13 Hp=[169,13,13,13,13,13,13,13,13,13,13,13,13]
p=17 rk(Hp)=16 #zp=17 Hp=[17,17,17,17,17,17,17,17,17,17,17,17,17,17,17,17]
p=19 rk(Hp)=19 #zp=19 Hp=[361,19,19,19,19,19,19,19,19,19,19,19,19,19,19,
                                                              19,19,19,19]
p=23 rk(Hp)=22 #zp=23 Hp=[23,23,23,23,23,23,23,23,23,23,23,23,23,23,23,
                                                     23,23,23,23,23,23,23]
el=2  n=2
p=3  rk(Hp)=4  #zp=5  Hp=[3,3,3,3]
p=5  rk(Hp)=9  #zp=9  Hp=[25,5,5,5,5,5,5,5,5]
p=7  rk(Hp)=12 #zp=13 Hp=[7,7,7,7,7,7,7,7,7,7,7,7]
p=11 rk(Hp)=21 #zp=21 Hp=[121,11,11,11,11,11,11,11,11,11,11,11,11,11,11,
                                                        11,11,11,11,11,11]
p=13 rk(Hp)=26 #zp=25 Hp=[169,169,13,13,13,13,13,13,13,13,13,13,13,13,13,
                                         13,13,13,13,13,13,13,13,13,13,13]
el=3  n=1
p=5  rk(Hp)=6  #zp=7  Hp=[5,5,5,5,5,5]
p=7  rk(Hp)=10 #zp=10 Hp=[49,7,7,7,7,7,7,7,7,7]
p=11 rk(Hp)=15 #zp=16 Hp=[11,11,11,11,11,11,11,11,11,11,11,11,11,11,11]
p=13 rk(Hp)=18 #zp=19 Hp=[13,13,13,13,13,13,13,13,13,13,13,13,13,13,13,
                                                               13,13,13]
p=17 rk(Hp)=24 #zp=25 Hp=[17,17,17,17,17,17,17,17,17,17,17,17,
                                    17,17,17,17,17,17,17,17,17,17,17,17]
p=19 rk(Hp)=27 #zp=28 Hp=[19,19,19,19,19,19,19,19,19,19,19,19,19,19,19,
                                    19,19,19,19,19,19,19,19,19,19,19,19]
el=3  n=2
p=5  rk(Hp)=18 #zp=19 Hp=[5,5,5,5,5,5,5,5,5,5,5,5,5,5,5,5,5,5]
p=7  rk(Hp)=28 #zp=28 Hp=[49,7,7,7,7,7,7,7,7,7,7,7,7,7,7,7,7,7,7,7,7,7,
                                                            7,7,7,7,7,7]
el=5 n=1
p=3  rk(Hp)=5  #zp=6  Hp=[3,3,3,3,3]
p=7  rk(Hp)=15 #zp=16 Hp=[7,7,7,7,7,7,7,7,7,7,7,7,7,7,7]
\end{verbatim}
\normalsize

\subsection{Probabilistic analysis for $p>2$}
Consider the following reflection theorem \cite[II.5.4.9.2, formula (4)]{Gra0}:

\begin{proposition} Let $L = K(\mu_p)$, $p \nmid N$; we have
${\rm rk}_p({\mathcal C}_K) = {\rm rk}_p(Y_{L,{\rm prim}}^\omega)$, where
$Y_{L,{\rm prim}}^\omega \subseteq Y_L^\omega :=\big( \{\alpha \in L^\times,
\, (\alpha) = {\mathfrak A}^p\} \! \cdot \! L^{\times p}/ L^{\times p} \big)^\omega$
is the $\omega$-component of the subset of $p$-primary elements $\alpha$ 
(i.e., such that $L(\sqrt[p]{\alpha})/L$ is unramified of degree $p$
and decomposed into a cyclic 
extension of $K$ in $H_K^{\rm nr}$). Thus ${\rm rk}_p({\mathcal C}_K) = 
{\rm rk}_p({\mathcal C}_L^\omega)$ or ${\rm rk}_p({\mathcal C}_L^\omega)-1$.
\end{proposition}

\begin{proof}
We have, from the general formula (loc. cit.): 
$${\rm rk}_p({\mathcal C}_K) = {\rm rk}_p({\mathcal C}_L^\omega) + 1 
- {\rm rk}_p(Y_L^\omega) +  {\rm rk}_p(Y_{L,{\rm prim}}^\omega). $$ 

Put $Y_L^\omega = \{\alpha_1, \ldots, \alpha_r\} \cup \{\zeta_p\}$
modulo $L^{\times p}$, the $\alpha_i$ being non-units and independent 
modulo $L^{\times p}$, and where $r$ is the $p$-rank of 
${\mathcal C}_L^\omega$. Since $\zeta_p$ is not $p$-primary, one gets
${\rm rk}_p({\mathcal C}_K) = {\rm rk}_p(Y_{L,{\rm prim}}^\omega)=
{\rm rk}_p(\langle \alpha_1, \ldots, \alpha_r \rangle_{\rm prim})$. 
Due to the $p$-adic action of $\omega$ on the $\alpha_i$, it is
immediate to deduce the last claim.
\end{proof}

The condition ${\rm rk}_p({\mathcal C}_K) \geq 1$ is then equivalent
to the existence of a $p$-primary $\alpha \in Y_L^\omega$ such that 
$(\alpha) = {\mathfrak A}^p$, with a non-principal ${\mathfrak A}$.
Program X gives cases where necessarily 
${\rm rk}_p({\mathcal C}_L) =r \geq 1$ (probably $r=1$, otherwise
we should have ${\rm rk}_p({\mathcal C}_K)= r$ or $ r - 1 \ne 0$);
one computes easily that the probability to have $\alpha$
$p$-primary (in a standard point of view) is $\frac{1}{p}$.

\smallskip
The computation of the class group of $L$ is rapidly out of reach and 
we have only been able to compute ${\mathcal C}_L$ for
$N = 3$ with $p=7$ giving ${\mathcal C}_L \simeq \Z/7\Z$;
we do not know $\alpha$ so that we cannot verify that it is not
$7$-primary (which is indeed the case since we know, from 
\S\,\ref{programR}, that ${\mathcal R}_K$ is not a $7$-adic unit).

\section{The $p$-torsion groups in the cyclotomic $\wh \Z$-extension $\wh \Q$}

Since there exist many fields $k=\Q(\ell^n)$ with non-trivial  
$p$-torsion groups ${\mathcal T}_k$, these groups remain subgroups
of ${\mathcal T}_K$ for any $K=\Q(N)$, extension of $k$ in $\wh \Q$, 
$N= \ell_1^{n_1} \cdots \ell_t^{n_t}$, 
and give larger groups. So we have essentially to compute 
${\mathcal T}_K^*$ (the relative submodule), 
product of the components ${\mathcal T}_K^{\theta}$
for $p$-adic characters $\theta$ given by the characters $\psi$
of order $N$ of $K$ (see \S\,\ref{Gmodules}).

\subsection{General program}\label{N}
The following completely general program uses the method of $p$-adic 
measure associated to the computation of Stickelberger's element for the 
composite conductor $f := p f_N$ of $K \Q(\mu_p)$, or $f=4 f_N$ if $p=2$; the Galois group 
$\G(\Q(\mu_{f})/\Q)$ is described by the program as direct product deduced 
from that of $(\Z/{f}\Z)^\times$ as usual, using a half system of representatives 
$\Sigma$ distinct from $[1, f/2]$ since it is not efficient to determine the Artin
automorphism of a representative $a \in [1, f/2]$.
All primes $p$ are tested, which will give some cases of annihilators
of degree $>1$ (hence primes $p$ of residue degree $>1$ in $\Q(\mu_N)$).

\smallskip
The choice of $c$, defining the multiplier 
$1 - c \cdot \big(\frac{\,\Q(\mu_f)}{c} \big)^{-1}$, gives some difficulty
for $N$ even since for $N$ odd, $c=2$ is always suitable
(except in the rare known cases where $2$ totally splits in $\Q(N)$,
giving integers $N$ out of reach). But $c$ must be chosen for each 
$p$ so that $\psi(c) \ne 1$, $\psi$ of order $N$, which increases 
dramatically the computing time since the Artin symbol of $c$ 
is not immediate; so, in the program, we only assume $c$ 
prime to $f$. Doing this, the case $\psi(c)=1$ may occur, 
giving $(1-\psi(c)) \cdot \hbox{$\frac{1}{2}$} L_p(1,\psi) = 0$
in the relation of Lemma \ref{psi}, while
$L_p(1,\psi) \ne 0$; but $\psi(c)$ is a $N$th root of unity and,
by assumption, $p\nmid N$, so $1-\psi(c)$ non-invertible 
modulo $p$ is equivalent to $\psi(c)=1$;
a unique example occurs for $N=10$ (line ${\sf {}^{\ast}}$ 
of the table) to be dropped since a direct verification via Program I
(\S\,\ref{T1}) does not give any solution $p$ in the selected interval.

\smallskip
The $p$-adic characters $\theta$ are defined using a factorization
modulo $p$ of the $N$th cyclotomic polynomial ${\sf Q}$ in polynomials
${\sf Rp[k]}$ in the list ${\sf Rp}$; then the program tests the condition 
${\sf S(x) \equiv 0 \pmod {Rp[k]}}$, where  ${\sf S(x)}$ represents 
${\mathcal A}'^c_K$ in the group algebra $\Z_p[x]$, ${\sf x}$ generating 
the Galois group.

\begin{lemma} \label{half}
When $\ell_1 = 2 \mid N$, let $f =: 2^{n_1+2} \cdot 
\ell_2^{n_2+1} \cdots  \ell_t^{n_t+1} \cdot p =: 2^{n_1+2}\cdot f'$ be the 
conductor of $K(\mu_p)$. One can neglect, in the summation over 
$a \in [1, f]$ defining ${\mathcal A}^c_{\Q(\mu_f)} = 
(1+s_{-1}) \,{\mathcal A}'^c_{\Q(\mu_f)}$, the component 
$\G(\Q(\mu_f)/k)$, where $k=\Q(2^{n_1})\Q(\mu_{f'})$.
When $N$ is odd one can use the representatives of
$\G(\Q(\mu_{\ell_1^{n_1+1}})/\Q)$ modulo its complex conjugation.
\end{lemma}

\begin{proof}
Let $s$ be the generator of $\G(\Q(\mu_f)/k)$,
$s_{-1}$ the complex conjugation and $\sigma_a:=
\big(\frac{\G(\Q(\mu_f)/\Q}{a} \big)$. 
Let $\Sigma := \{a \in [1, f], \  \sigma_a \in \G(\Q(\mu_f)/\Q(\mu_4)) \}$ 
be the set of representatives in $[1, f]$ used by the program,
so that $[1, f] = \Sigma \cup \ov \Sigma$, where $\ov \Sigma$ represents
$s \!\cdot \! \G(\Q(\mu_f)/\Q(\mu_4))$; then ${\mathcal A}^c_{\Q(\mu_f)} 
= \sum_{a \in \Sigma \cup \ov \Sigma} \lambda_a(c) a^{-1} \sigma_a$.
Then $\Sigma \cup (f-\Sigma)$ is a partition of $[1, f]$ ($s$ and $s_{-1}$ 
project on $\G(\Q(\mu_4)/\Q)$); thus ${\mathcal A}^c_{\Q(\mu_f)} = 
\sum_{a \in \Sigma \cup (f-\Sigma)} \lambda_a(c) a^{-1} \sigma_a$.
Put ${\mathcal B}' := \sum_{a \in \Sigma} \lambda_a(c) a^{-1} \sigma_a$; 
then $(1+s_{-1}) \cdot {\mathcal B}' \equiv {\mathcal A}^c_{\Q(\mu_f)}\!\! \pmod p$
(Lemma \ref{sprime}).
Since ${\mathcal A}^c_{\Q(\mu_f)} = (1+s_{-1})\, {\mathcal A}'^c_{\Q(\mu_f)}$,
one deduces that ${\mathcal A}'^c_{\Q(\mu_f)} \equiv {\mathcal B}'$
modulo the ideal $((1-s_{-1}), p)$; whence the claim since $\psi$ is even.
The case of odd $N$ is obvious.
\end{proof}

Then we shall perform some verifications by using the basic Programs I, II, 
\S\,\ref{nontrivial}, when computation via ${\sf K=bnfinit(P)}$ is possible,
which holds only for small conductors contrary to the present method
allowing computations up to $N=400$ and beyond, with large primes $p$ 
without any more memory. But the standard method gives the structure of 
${\mathcal T}_K$ contrary to the present one, only
giving the annihilator of ${\mathcal T}^*_K$ modulo $p$.

\smallskip
The instruction ${\sf if(Mod(p,N)!=1,next)}$
only considers primes $p$ totally split in $\Q(\mu_N)$ (faster, but 
eliminates cases with annihilators of degree $>1$); the instructions
which follows only consider primes $p$ totally split in $K$
(these parts have been suppressed in the writing but may be restored).
Otherwise, the program tests all primes (case of the table obtained below).
The program may give the polynomial ${\sf P}$ defining $K=\Q(N)$.

\smallskip
To simplify the use by the reader and to reduce the 
execution time, the program considers two cases 
($N$ even and $N$ odd) and deals with 4 possibilities
corresponding to the number ${\sf dim}$ of divisors of $N$
so that the program can work for $2 \leq N \leq 2309$.
The bound ${\sf Bp=floor(2*10^5/N)}$ for $p$ may be modified at will.
The variables ${\sf fN}$ and ${\sf f=p*fN}$ denote the conductors of $K$ 
and $K(\mu_p)$, respectively. \par
\footnotesize
\begin{verbatim}
PROGRAM XI. ANNIHILATORS OF T*, FOR ALL N >1
{BN=500;for(N=2,BN,Bp=floor(2*10^5/N);dim=omega(N);Q=polcyclo(N);
Lq=List;LQ=List;Lh=List;LH=List;LN=List;divN=factor(N);
D=component(divN,1);Exp=component(divN,2);
if(Mod(N,2)==0,delta=1;fN=1;
q1=D[1]^Exp[1];listput(Lq,q1,1);Q1=4*q1;listput(LQ,Q1,1);
N1=N/q1;listput(LN,N1,1);fN=fN*Q1;
h1=Mod(5,Q1);listput(Lh,h1,1);listput(LH,lift(h1),1);
for(i=2,dim,qi=D[i]^Exp[i];listput(Lq,qi,i);Qi=qi*D[i];
listput(LQ,Qi,i);fN=fN*Qi;Ni=N/qi;listput(LN,Ni,i);
hi=znprimroot(Qi);listput(Lh,hi,i);listput(LH,lift(Lh[i]),i)));
if(Mod(N,2)!=0,C=2;delta=2;fN=1;
for(i=1,dim,qi=D[i]^Exp[i];listput(Lq,qi,i);Qi=qi*D[i];
listput(LQ,Qi,i);fN=fN*Qi;Ni=N/qi;listput(LN,Ni,i);
hi=znprimroot(LQ[i]);listput(Lh,hi,i);listput(LH,lift(hi),i)));
\\polynomial of Q(N):
\\if(Mod(N,2)==0,P=x;for(i=1,Exp[1],P=P^2-2);for(i=2,dim,
\\P=polcompositum(P,polsubcyclo(LQ[i],Lq[i]))[1]);print("N=",N," P=",P));
\\if(Mod(N,2)!=0,P=x;for(i=1,dim,
\\P=polcompositum(P,polsubcyclo(LQ[i],Lq[i]))[1]);print("N=",N," P=",P));
if(dim>=1,E1=eulerphi(LQ[1])/delta);if(dim>=2,E2=eulerphi(LQ[2]));
if(dim>=3,E3=eulerphi(LQ[3]));if(dim>=4,E4=eulerphi(LQ[4]));
forprime(p=3,Bp,if(Mod(N,p)==0,next);
\\Specifies the primes p totally split in Q(mu_N):
\\if(Mod(p,N)!=1,next);
\\Specifies the primes p totally split in Q(N):
\\if(Mod(N,2)==0,w1=valuation(p^2-1,2);if(Exp[1]+3>w1,next);
\\for(j=2,dim,wj=valuation(p^(D[j]-1)-1,D[j]);if(Exp[j]+1>wj,next(2))));
\\if(Mod(N,2)!=0,
\\for(j=1,dim,wj=valuation(p^(D[j]-1)-1,D[j]);if(Exp[j]+1>wj,next(2))));
g=znprimroot(p);G=lift(g);gm=g^-1;
f=p*fN;M=f/p;E=lift(Mod((1-G)*p^-1,M));G=G+E*p;g=Mod(G,f);
for(j=1,dim,M=f/LQ[j];E=lift(Mod((1-LH[j])*LQ[j]^-1,M));
H=LH[j]+E*LQ[j];listput(Lh,Mod(H,f),j));
if(Mod(N,2)==0,Cc=2;while(gcd(Cc,f)!=1,Cc=Cc+1);C=Cc;cm=Mod(C,f)^-1);
if(Mod(N,2)!=0,C=2;cm=Mod(C,f)^-1);
F=factor(Q+O(p));R=lift(component(F,1));d=matsize(F)[1];Rp=List;
for(j=1,d,listput(Rp,R[j]*Mod(1,p),j));Qp=Q*Mod(1,p);gg=1;ggm=1;hh=1;S=0;
if(dim==1,
for(u1=1,E1,hh=hh*Lh[1];t=0;
for(v=1,p-1,gg=gg*g;ggm=ggm*gm;a=lift(hh*gg);A=lift(a*cm);
t=t+(A*C-a)/f*ggm);e=lift(Mod(u1*LN[1],N));
S=S+lift(t)*x^e);S=S*Mod(1,p);S=lift(Mod(S,Qp));for(k=1,d,Rk=Rp[k];
if(Mod(S,Rk)==0,print("N=",N," p=",p," annihilator = ",Rk))));
if(dim==2,
for(u1=1,E1,hh=hh*Lh[1];for(u2=1,E2,hh=hh*Lh[2];
t=0;for(v=1,p-1,gg=gg*g;ggm=ggm*gm;
a=lift(hh*gg);A=lift(a*cm);
t=t+(A*C-a)/f*ggm);e=lift(Mod(u1*LN[1]+u2*LN[2],N));
S=S+lift(t)*x^e));S=S*Mod(1,p);S=lift(Mod(S,Qp));for(k=1,d,Rk=Rp[k];
if(Mod(S,Rk)==0,print("N=",N," p=",p," annihilator = ",Rk))));
if(dim==3,
for(u1=1,E1,hh=hh*Lh[1];for(u2=1,E2,hh=hh*Lh[2];
for(u3=1,E3,hh=hh*Lh[3];t=0;
for(v=1,p-1,gg=gg*g;ggm=ggm*gm;a=lift(hh*gg);A=lift(a*cm);
t=t+(A*C-a)/f*ggm);e=lift(Mod(u1*LN[1]+u2*LN[2]+u3*LN[3],N));
S=S+lift(t)*x^e)));S=S*Mod(1,p);S=lift(Mod(S,Qp));for(k=1,d,Rk=Rp[k];
if(Mod(S,Rk)==0,print("N=",N," p=",p," annihilator = ",Rk))));
if(dim==4,
for(u1=1,E1,hh=hh*Lh[1];for(u2=1,E2,hh=hh*Lh[2];
for(u3=1,E3,hh=hh*Lh[3];for(u4=1,E4,hh=hh*Lh[4];t=0;
for(v=1,p-1,gg=gg*g;ggm=ggm*gm;a=lift(hh*gg);A=lift(a*cm);
t=t+(A*C-a)/f*ggm);e=lift(Mod(u1*LN[1]+u2*LN[2]+u3*LN[3]+u4*LN[4],N));
S=S+lift(t)*x^e))));S=S*Mod(1,p);S=lift(Mod(S,Qp));for(k=1,d,Rk=Rp[k];
if(Mod(S,Rk)==0,print("N=",N," p=",p," annihilator = ",Rk))))))}
\end{verbatim}
\normalsize

\subsection{Table of non-trivial ${\mathcal T}_{\Q(N)}^\theta$}\label{table}
Let $\sigma$ be a generator of $\G(K/\Q)$, $K:=\Q(N)$, $N \geq 2$.
To simplify notations, ${\sf (a, p)}$ means ${\sf Mod(a, p)}$;
an annihilator $f(x)$ gives the Galois action $\tau^{f(\sigma)} =1$ 
in ${\mathcal T}_K/{\mathcal T}_K^p$ and defines a $p$-adic 
character $\theta$ ($\theta$ above $\psi$ of order $N$) for which 
${\mathcal T}_K^\theta \ne 1$ (recall that the order and the 
annihilation of ${\mathcal T}_K^\theta$ is given by $\frac{1}{2} L_p(1,\psi)$). 
For instance, the two informations:

\smallskip
\centerline{${\sf N=2,\ p=13,\ (1,13)*x+(1,13)}$ and ${\sf N=4,\ p=13,\ (1,13)*x+(5,13)}$}

\smallskip
\noindent
are related to characters $\theta$ of orders $2$ and $4$, for which the 
${\mathcal T}_{\Q(N)}^{\theta}$ are non-trivial
(see the complete structure of ${\mathcal T}_{\Q(4)}$ in the table
of Program II, \S\,\ref{T1}).\par
\footnotesize
\begin{verbatim}
             ANNIHILATORS mod p                    ANNIHILATORS mod p
 N=2   p=13   (1,13)*x+(1,13)         N=128 p=641  (1,641)*x+(287,641)
 N=2   p=31   (1,31)*x+(1,31)         N=129 p=257  (1,257)*x^2
 N=3   p=7    (1,7)*x+(5,7)                         +(81,257)*x+(1,257)
 N=3   p=73   (1,73)*x+(9,73)         N=136 p=137  (1,137)*x+(35,137)
 N=4   p=13   (1,13)*x+(5,13)         N=138 p=139  (1,139)*x+(31,139)
 N=4   p=29   (1,29)*x+(12,29)        N=140 p=29   (1,29)*x^2
 N=4   p=37   (1,37)*x+(31,37)                        +(3,29)*x+(5,29)
 N=5   p=11   (1,11)*x+(7,11)         N=144 p=433  (1,433)*x+(292,433)
 N=5   p=11   (1,11)*x+(8,11)         N=153 p=307  (1,307)*x+(178,307)
 N=6   p=7    (1,7)*x+(2,7)           N=155 p=311  (1,311)*x+(203,311)
 N=6   p=13   (1,13)*x+(9,13)         N=156 p=157  (1,157)*x+(80,157)
 N=6   p=43   (1,43)*x+(36,43)        N=172 p=173  (1,173)*x+(143,173)
 N=8   p=3    (1,3)*x^2+(1,3)*x+(2,3) N=174 p=349  (1,349)*x+(16,349)
 N=8   p=521  (1,521)*x+(206,521)     N=178 p=179  (1,179)*x+(129,179)
*N=10  p=3    (1,3)*x^4+(2,3)*x^3     N=190 p=761  (1,761)*x+(94,761)
             +(1,3)*x^2+(2,3)*x+(1,3) N=191 p=383  (1,383)*x+(315,383)
 N=12  p=13   (1,13)*x+(7,13)         N=191 p=383  (1,383)*x+(360,383)
 N=14  p=113  (1,113)*x+(106,113)     N=192 p=193  (1,193)*x+(115,193)
 N=15  p=31   (1,31)*x+(11,31)        SOLUTIONS p<10*N SPLIT IN Q(mu_N):
 N=15  p=31   (1,31)*x+(22,31)        N=210 p=211  (1,211)*x+(59,211)
 N=15  p=241  (1,241)*x+(81,241)      N=210 p=211  (1,211)*x+(154,211)
 N=15  p=1291 (1,1291)*x+(958,1291)   N=215 p=431  (1,431)*x+(74,431)
 N=17  p=239  (1,239)*x+(172,239)     N=215 p=1721 (1,1721)*x+(162,1721)
 N=18  p=37   (1,37)*x+(33,37)        N=225 p=1801 (1,1801)*x+(1536,1801)
 N=22  p=397  (1,397)*x+(16,397)      N=226 p=227  (1,227)*x+(160,227)
 N=22  p=2729 (1,2729)*x+(1268,2729)  N=230 p=691  (1,691)*x+(345,691)
 N=23  p=47   (1,47)*x+(19,47)        N=230 p=1381 (1,1381)*x+(144,1381)
 N=25  p=101  (1,101)*x+(21,101)      N=234 p=1171 (1,1171)*x+(988,1171)
 N=25  p=1151 (1,1151)*x+(744,1151)   N=236 p=1181 (1,1181)*x+(939,1181)
 N=25  p=2251 (1,2251)*x+(1033,2251)  N=240 p=241  (1,241)*x+(110,241)
 N=27  p=109  (1,109)*x+(20,109)      N=242 p=2179 (1,2179)*x+(1976,2179)
 N=28  p=701  (1,701)*x+(338,701)     N=249 p=499  (1,499)*x+(242,499)
 N=29  p=59   (1,59)*x+(56,59)        N=261 p=2089 (1,2089)*x+(1080,2089)
 N=30  p=1831 (1,1831)*x+(261,1831)   N=265 p=1061 (1,1061)*x+(919,1061)
 N=33  p=397  (1,397)*x+(136,397)     N=276 p=277  (1,277)*x+(272,277)
 N=38  p=2357 (1,2357)*x+(659,2357)   N=281 p=563  (1,563)*x+(551,563)
 N=39  p=157  (1,157)*x+(44,157)      N=284 p=2557 (1,2557)*x+(1876,2557)
 N=40  p=41   (1,41)*x+(22,41)        N=288 p=1153 (1,1153)*x+(428,1153)
 N=40  p=41   (1,41)*x+(30,41)        N=288 p=1153 (1,1153)*x+(577,1153)
 N=40  p=41   (1,41)*x+(35,41)        N=290 p=1451 (1,1451)*x+(135,1451)
 N=43  p=173  (1,173)*x+(41,173)      N=292 p=877  (1,877)*x+(405,877)
 N=45  p=541  (1,541)*x+(336,541)     N=293 p=587  (1,587)*x+(323,587)
 N=47  p=283  (1,283)*x+(27,283)      N=296 p=593  (1,593)*x+(447,593)
 N=48  p=193  (1,193)*x+(28,193)      N=296 p=1481 (1,1481)*x+(444,1481)
 N=50  p=101  (1,101)*x+(88,101)      N=303 p=607  (1,607)*x+(59,607)
 N=50  p=251  (1,251)*x+(123,251)     N=303 p=607  (1,607)*x+(564,607)
 N=50  p=1201 (1,1201)*x+(493,1201)   N=306 p=307  (1,307)*x+(7,307)
 N=52  p=53   (1,53)*x+(12,53)        N=306 p=919  (1,919)*x+(81,919)
 N=52  p=53   (1,53)*x+(21,53)        N=307 p=1229 (1,1229)*x+(121,1229)
 N=52  p=53   (1,53)*x+(27,53)        N=309 p=619  (1,619)*x+(32,619)
 N=52  p=157  (1,157)*x+(128,157)     N=315 p=631  (1,631)*x+(346,631)
 N=54  p=163  (1,163)*x+(21,163)      N=321 p=643  (1,643)*x+(520,643)
 N=56  p=13   (1,13)*x^2              N=324 p=2269 (1,2269)*x+(1878,2269)
               +(5,13)*x+(5,13)       N=324 p=2593 (1,2593)*x+(1526,2593)
 N=60  p=61   (1,61)*x+(43,61)        N=328 p=2953 (1,2953)*x+(2160,2953)
 N=63  p=379  (1,379)*x+(302,379)     N=330 p=331  (1,331)*x+(46,331)
 N=64  p=193  (1,193)*x+(160,193)     N=330 p=331  (1,331)*x+(110,331)
 N=66  p=1321 (1,1321)*x+(617,1321)   N=335 p=2011 (1,2011)*x+(919,2011)
 N=67  p=269  (1,269)*x+(176,269)     N=340 p=1021 (1,1021)*x+(417,1021)
 N=67  p=269  (1,269)*x+(208,269)     N=340 p=1021 (1,1021)*x+(993,1021)
 N=69  p=829  (1,829)*x+(532,829)     N=340 p=2381 (1,2381)*x+(1143,2381)
 N=70  p=71   (1,71)*x+(40,71)        N=344 p=1721 (1,1721)*x+(939,1721)
 N=70  p=211  (1,211)*x+(76,211)      N=345 p=1381 (1,1381)*x+(502,1381)
 N=72  p=73   (1,73)*x+(28,73)        N=346 p=2423 (1,2423)*x+(2301,2423)
 N=80  p=241  (1,241)*x+(124,241)     N=348 p=349  (1,349)*x+(132,349)
 N=81  p=487  (1,487)*x+(287,487)     N=352 p=353  (1,353)*x+(238,353)
 N=83  p=499  (1,499)*x+(312,499)     N=358 p=359  (1,359)*x+(111,359)
 N=84  p=757  (1,757)*x+(685,757)     N=358 p=359  (1,359)*x+(240,359)
 N=86  p=431  (1,431)*x+(145,431)     N=362 p=1087 (1,1087)*x+(172,1087)
 N=87  p=349  (1,349)*x+(157,349)     N=363 p=1453 (1,1453)*x+(1416,1453)
 N=87  p=523  (1,523)*x+(62,523)      N=363 p=2179 (1,2179)*x+(18,2179)
 N=88  p=353  (1,353)*x+(17,353)      N=368 p=3313 (1,3313)*x+(2536,3313)
 N=93  p=373  (1,373)*x+(307,373)     N=375 p=751  (1,751)*x+(335,751)
 N=95  p=191  (1,191)*x+(132,191)     N=382 p=383  (1,383)*x+(23,383)
 N=95  p=191  (1,191)*x+(137,191)     N=386 p=1931 (1,1931)*x+(1315,1931)
 N=99  p=991  (1,991)*x+(91,991)      N=388 p=389  (1,389)*x+(233,389)
 N=99  p=991  (1,991)*x+(818,991)     N=388 p=1553 (1,1553)*x+(421,1553)
 N=100 p=199  (1,199)*x^2             N=388 p=1553 (1,1553)*x+(464,1553)
             +(173,199)*x+(1,199)     N=395 p=2371 (1,2371)*x+(2137,2371)  
 N=101 p=607  (1,607)*x+(277,607)     N=400 p=401  (1,401)*x+(294,401)
 N=101 p=607  (1,607)*x+(514,607)     N=401 p=3209 (1,3209)*x+(154,3209)
 N=102 p=103  (1,103)*x+(83,103)      N=401 p=4813 (1,4813)*x+(3529,4813) 
 N=102 p=103  (1,103)*x+(97,103)      N=405 p=811  (1,811)*x+(645,811)
 N=104 p=937  (1,937)*x+(609,937)     N=407 p=3257 (1,3257)*x+(894,3257)
 N=106 p=107  (1,107)*x+(39,107)      N=407 p=3257 (1,3257)*x+(2268,3257)
 N=106 p=107  (1,107)*x+(61,107)      N=408 p=409  (1,409)*x+(370,409)
 N=107 p=857  (1,857)*x+(263,857)     N=412 p=1237 (1,1237)*x+(387,1237)
 N=108 p=109  (1,109)*x+(24,109)      N=420 p=421  (1,421)*x+(367,421)
 N=111 p=223  (1,223)*x+(176,223)     N=422 p=2111 (1,2111)*x+(615,2111)
 N=115 p=461  (1,461)*x+(87,461)      N=427 p=1709 (1,1709)*x+(922,1709)
 N=115 p=461  (1,461)*x+(103,461)     N=428 p=857  (1,857)*x+(31,857)
 N=118 p=709  (1,709)*x+(27,709)      N=429 p=3433 (1,3433)*x+(702,3433)
 N=124 p=5    (1,5)*x^3+(2,5)*x^2     N=430 p=1291 (1,1291)*x+(1091,1291)
                   +(2,5)*x+(3,5)     N=431 p=863  (1,863)*x+(406,863)
 N=124 p=373  (1,373)*x+(139,373)     N=431 p=863  (1,863)*x+(754,863)
 N=124 p=373  (1,373)*x+(340,373)     N=432 p=3889 (1,3889)*x+(2110,3889)
 N=126 p=379  (1,379)*x+(165,379)     N=442 p=443  (1,443)*x+(325,443)
 N=128 p=257  (1,257)*x+(113,257)     N=443 p=887  (1,887)*x+(226,887)
\end{verbatim}
\normalsize

\begin{remark} 
One knows a unique example of the form $p \mid \order {\mathcal C}_{\Q(N)}$
for $p \nmid N$ with the data $p=107$, $N =2 \cdot 53$ (Aoki--Fukuda \cite{AF}).
This value of $N$ does appear in the table with two annihilators
${\sf Mod(1,107)*x+Mod(39,107)}$,  ${\sf Mod(1,107)*x+Mod(61,107)}$; we ignore
the contributions for ${\mathcal C}_K$ and ${\mathcal R}_K$. 

\smallskip
We observe that many values of $N$ give more than one annihilator; 
they are perhaps good candidates for similar examples, even if all 
combinations are possible (including the cases of a unique annihilator).

\smallskip
For instance, the fields $K=\Q(5)$ and $\Q(15)$ give rise to two annihilators, 
which is specified by the structure ${\mathcal T}_K \simeq (\Z/11\Z)^2$ and 
${\mathcal T}_K \simeq (\Z/31\Z)^2$, but we compute that ${\mathcal C}_K=1$
in these cases. In \cite{Ho4} it is proved that for $\ell < 131, 109, 101$, the 
$p$-class group of $\Q(\ell^\infty)$ is trivial for $p=7, 11, 13$, respectively. 
This confirms that for $N=5$, $p=11$, ${\mathcal T}_K
= {\mathcal R}_K \simeq (\Z/11\Z)^2$.
The next examples in the table are $(N, p)=(40, 41)$, $(52, 53)$, $(67, 269)$, 
$(95,191)$, $(99, 991)$, etc. It would be interesting to test these fields.
\end{remark}

\smallskip
Let's give some verifications, using Program I \S\,\ref{T1}, computing independently 
the structure of ${\mathcal T}_K$; only very small $N$ can be tested because 
of the instructions ${\sf K=bnfinit(P)}$ and ${\sf KpEx=bnrinit(K,p^{Ex})}$, where the 
defining polynomial ${\sf P}$ may be obtained with Program XI above; the structure 
obtained for ${\mathcal T}_K$ depends on that of the subfields of $K$, while 
Program XI only gives the $p$-rank of ${\mathcal T}^*_K$. 

\smallskip
For instance,
the case $N=8$, $p=3$, with the annihilator $3\,x^2 + x +2 \pmod p$
is the first annihilator of degree $>1$; since (from the table) ${\mathcal T}_K$
is annihilated by the relative norm $x^4+1 \equiv (x^2+x+2)(x^2+2x+2)\!\!
\pmod 3$ and since~$3$ is totally inert, the result gives at least a $3$-rank 2.
This is validated as ${\sf N=8,\  p=3,\  rk(T)=2,\  T=List([3,3])}$.\par
\footnotesize
\begin{verbatim}
PROGRAM XII. STRUCTURE OF T IN SOME Q(N) AND VERIFICATIONS WITH PROGRAM I:
Field K=Q(5)  T=[11,11]
Field K=Q(6)  T=[7,7] T=[13,13] T=[31] T=[43] T=[73]
Field K=Q(12) T=[9,9] T=[7,7] T=[169,169,13,13] T=[29] T=[31] T=[37] 
                                                       T=[43] T=[73]
Field K=Q(14) T=[13] T=[31] T=[113]
Field K=Q(15) T=[7] T=[11,11] T=[31,31] T=[73]
Field K=Q(21) T=[49,7]
Field K=Q(30) T=[7,7] T=[11,11] T=[13,13] T=[31,31,31] T=[43] T=[73]
Field K=Q(42) T=[49,49,7,7] T=[13,13]
\end{verbatim}
\normalsize     

 The composite $K=\Q(42)$ has some interest for $p=7$ since ${\mathcal T}_K \simeq (\Z/7\Z)^2$; 
so we know that ${\mathcal T}_K^{\G(K/k)} \simeq {\mathcal T}_k$, where $k=\Q(6)$;
but with ${\mathcal T}_K  \simeq (\Z/7\Z)^2 \times  (\Z/7^2\Z)^2$,
showing that for $p$-ramification aspects, genus theory gives often increasing 
$p$-torsion groups contrary to $p$-class groups as we shall see in the next Section.
Since $\No_{K/k}({\mathcal T}_K) = {\mathcal T}_k$, we have
${\mathcal T}_K^* \simeq (\Z/7^2\Z)^2$.  Note that the case $N=42$
does not appear in the table because of the condition $p \nmid N$
which will be the framework of genus theory in $\Q(N)\Q(p^\infty)$.

\section{Genus theory and $p$-class groups in $\wh \Q$}\label{genus}
We consider, in the cyclotomic $\wh \Z$-extension $\wh \Q$, any subfield 
of finite or infinite degree, and fix a prime~$p$ (see \cite{Mor3} for 
analytic results of non-divisibility in this context).

\subsection{Definition of $\wh \Q^*$}
The pro-cyclic extension 
$\wh \Q$ is the direct composite over $\Q$ of $\Q(p^\infty)$ and 
the composite $\wh \Q^*$ of all the $\Q(\ell^\infty)$, for $\ell \ne p$.
Two cases then arise: that of the $p$-class groups of $K=\Q(N)$ when 
$p \nmid N$ and the case of fields written as composite $K_m = K \Q(p^m)$, 
$K \subset \wh \Q^*$, $m\geq 1$.

\smallskip
In the first case, we are in a generalization of Weber's problem.
In the second one the problem is related to genus theory,
whence to Greenberg's conjecture \cite{Gree}, for which one very strongly 
admits that $\order {\mathcal C}_{K \Q(p^m)}$ is constant 
for all $m \gg 0$ (i.e., the invariants $\lambda, \mu$ of $K$ 
for the prime $p$ are zero); see for instance \cite{FK1,Gra10,Jau2} 
for some developments. But we have:

\begin{proposition} Let $K = \Q(N) \subset \wh \Q^*$ (i.e., $p \nmid N$); let
$K_m=K \Q(p^m)$ for any $m \geq 0$. Then, under Leopoldt's conjecture, 
${\mathcal T}_{K_m}=1$, if and only if ${\mathcal T}_K=1$.
Thus, a necessary condition to get ${\mathcal C}_{K_m} \ne 1$ (for some $m \geq 0$)
is ${\mathcal T}_K \ne 1$, which brings into play the general Table \ref{table}.
\end{proposition}

\begin{proof} Since $K_m/K$ is a $p$-ramified $p$-extension, the claim 
comes from the fixed points formula giving ${\mathcal T}_{K_m}^{\G(K_m/K)}
\simeq {\mathcal T}_K$ (\cite[Theorem IV.3.3]{Gra0}, \cite[Proposition 6]{Gra3}, 
\cite[Appendix A.4.2]{Gra9}).
\end{proof}

We shall study the reciprocal aspects in the next subsection
to result ultimately in Theorem \ref{thmfond}.

\subsection{The $p$-class group of $K_m$ -- 
Fundamental relation with ${\mathcal R}_K$} \label{diagram}
The analog of Weber's problem in $\wh \Q$ is, a priori, doubtful
because of Chevalley's formula in an extension $K_m/K$, 
$K \subset \wh \Q^*$, $K_m = K \Q(p^m)$: 
$$\order ({\mathcal C}_{K_m}^{\rm res})^{\G(K_m/K)} = 
\order {\mathcal C}_K^{\rm res} \cdot \ffrac{p^{m\, (s_p-1)}}
{(E^{\rm pos}_K : E_K^{\rm pos} \cap \No_{K_m/K} (K_m^\times) )}, $$ 
where $s_p := \order S$, the number of $p$-places.
So ${\mathcal C}_{K_m}^{\rm res}=1$ as soon as 
${\mathcal C}_K^{\rm res} = 1$ and $s_p=1$. 
If $s_p > 1$, the right factor may be a power of $p$ 
only depending on the normic properties of $E_K^{\rm pos}$
in $K_m/K$. 

\smallskip
Consider the two diagrams \cite[\S\,III.4.4.1]{Gra0} and 
\cite[Diagrams 2 and 3]{Gra10} (under the Leopoldt conjecture),
where $K^{\rm ab}$ is the maximal abelian pro-$p$-extension of $K$:
\unitlength=0.86cm
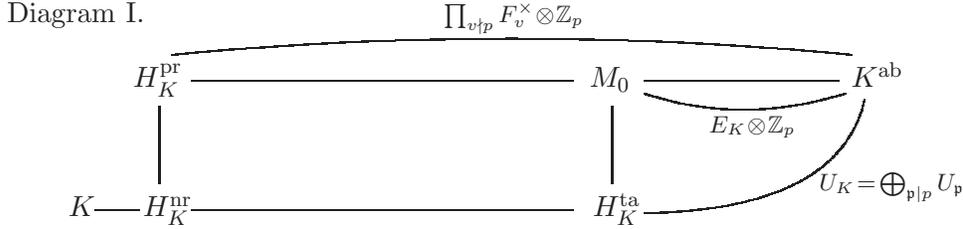
$$\vbox{\hbox{\hspace{-1.5cm} \vspace{-0.2cm}
\begin{picture}(11.5,3.7)
\put(-1.35,3.4){$\hbox{Diagram I.}$}
\put(8.5,2.50){\line(1,0){3.0}}
\put(1.5,2.50){\line(1,0){5.9}}
\put(1.5,0.50){\line(1,0){5.9}}
\put(1.0,0.9){\line(0,1){1.20}}
\put(0.0,0.5){\line(1,0){0.70}}
\put(8.00,0.9){\line(0,1){1.20}}
\bezier{400}(1.2,2.9)(6.45,3.4)(11.7,2.9)
\put(5.4,3.4){\ft ${\prod_{v \nmid p}{F_v^\times \! \otimes\! \Z_p}}$}
\bezier{280}(8.5,0.45)(11.5,0.5)(11.9,2.2)
\put(11.2,0.8){\ft $U_K \!=\! \bigoplus_{{\mathfrak p}\mid p}U_{\mathfrak p}$}
\bezier{250}(8.5,2.3)(10.0,1.8)(11.6,2.3)
\put(9.5,1.7){\ft $E_K \! \otimes\! \Z_p$}
\put(11.7,2.4){$K^{\rm ab}$}
\put(7.65,2.4){$M_0$}
\put(0.6,2.4){$H_K^{\rm pr}$}
\put(7.7,0.4){$H_K^{\rm ta}$}
\put(0.75,0.4){$H_K^{\rm nr}$}
\put(-0.4,0.4){$K$}
\end{picture}   }} $$
\unitlength=1.0cm
\noindent
where $F_v$ is the residue field of the tame place $v$ (finite or infinite).
We know that, in an idelic framework, the fixed field of
$U_K = \bigoplus_{{\mathfrak p} \mid p} U_{\mathfrak p}$ is
the maximal tame sub-extension $H_K^{\rm ta}$, since each 
$U_{\mathfrak p}$ is the inertia group of ${\mathfrak p}$ in $K^{\rm ab}/K$.
From Diagram I, the restriction of $U_K$ to ${\rm Gal}(H_K^{\rm pr}/K)$ 
is ${\rm Gal}(H_K^{\rm pr}/H_K^{\rm nr}) \simeq U_K/\iota_p(E_K \otimes \Z_p)$ 
whose torsion group is ${\rm Gal}(H_K^{\rm pr}/K_\infty H_K^{\rm nr})$.
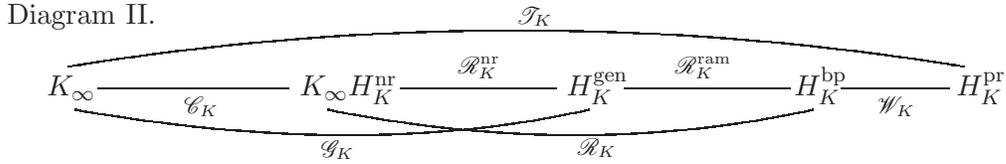
\unitlength=0.96cm
$$\vbox{\hbox{\hspace{-2.5cm} \vspace{0.4cm}
\begin{picture}(11.5,1.7)
\put(-0.05,1.4){$\hbox{Diagram II.}$}
\bezier{550}(0.8,0.8)(7.0,1.8)(13.2,0.8)
\put(7.0,1.4){\ft  ${\mathcal T}_K$}
\bezier{350}(4.4,0.2)(7.9,-0.5)(11.1,0.2)
\put(7.85,-0.42){\ft  ${\mathcal R}_K$}
\put(6.2,0.7){\ft  ${\mathcal R}_K^{\rm nr}$}
\put(9.2,0.7){\ft  ${\mathcal R}_K^{\rm ram}$}
\put(2.4,0.15){\ft  ${\mathcal C}_K$}
\put(5.4,0.50){\line(1,0){2.15}}
\put(7.7,0.4){$H_K^{\rm gen}$}
\put(8.5,0.50){\line(1,0){2.3}}
\put(11.5,0.50){\line(1,0){1.5}}
\put(10.85,0.4){$H_K^{\rm bp}$}
\put(12.0,0.15){\ft ${\mathcal W}_K$}
\put(1.2,0.50){\line(1,0){2.65}}
\put(4.0,0.4){$K_\infty H_K^{\rm nr}$}
\put(0.5,0.4){$K_\infty$}
\put(13.1,0.4){$H_K^{\rm pr}$}
\bezier{350}(0.9,0.2)(4.9,-0.5)(8.0,0.2)
\put(4.3,-0.45){\ft  ${\mathcal G}_K$}
\end{picture} }} $$
\unitlength=1.0cm
with ${\mathcal G}_K := \G(H_K^{\rm gen}/K_\infty)$, where
$H_K^{\rm gen}$ is the union, over $m$, of the genus fields 
$H_{K_m/K}$ (maximal abelian $p$-extensions of $K$, 
unramified over $K_m$; then $[H_{K_m/K} : K_m] = 
\order  {\mathcal C}_{K_m}^g$ for $g := \G(K_m/K)$); it follows that $H_K^{\rm gen}$ is 
the maximal unramified extension of $K_\infty$ 
in $H_K^{\rm pr}$ \cite[Proposition 3.6]{Gra10} and that:
\begin{equation}\label{genres}
\order {\mathcal G}_K = \order {\mathcal C}_K\, \ffrac{p^{m\,(s_p-1)}}
{(E^{\rm pos}_K : E_K^{\rm pos} \cap \No_{K_m/K} (K_m^\times))} =
\order {\mathcal C}_{K_m}^g,\ \hbox{for $m$ large enough}
\end{equation}

The inertia groups, in $H_K^{\rm pr}/K_\infty$, of the $p$-places are the torsion 
parts of the images of the $U_{\mathfrak p}$, then are isomorphic to 
${\rm tor}_{\Z_p} (U_{\mathfrak p}/\iota_p(E_K\otimes \Z_p) \cap U_{\mathfrak p})$.
So, the subgroup ${\mathcal I}_K$ of ${\mathcal T}_K$ generated by these inertia groups
fixes $H_K^{\rm gen}$. We then have the following result (under Leopoldt's conjecture) 
which will be fundamental for the search of non-trivial ${\mathcal C}_{K_m}^{\rm res}$:

\begin{lemma} \label{split}
Let $p$ be totally split in $K=\Q(N)$, $N \ne 1$.
The Galois group ${\mathcal I}_K$ generated by the inertia groups 
${\rm tor}_{\Z_p}(U_{\mathfrak p}/\iota_p(E_K\otimes \Z_p) \cap U_{\mathfrak p})$
is isomorphic to ${\mathcal W}_K$, trivial for $p\ne 2$, isomorphic to 
$\F_2^{N-1}$ for $p=2$ (Lemma \ref{lemmaW}). Thus
${\mathcal R}_K^{\rm ram}=1$, ${\mathcal R}_K^{\rm nr}={\mathcal R}_K$
and $\order {\mathcal G}_K = \order {\mathcal C}_K \order {\mathcal R}_K$
(cf. Diagram II).
\end{lemma}

\begin{proof} Let $\varepsilon \in E_K\otimes \Z_p$ be such that the 
diagonal image $\iota_p(\varepsilon)$ in $U_K$ is $\iota_p(\varepsilon) = 
(\iota_{\mathfrak p}(\varepsilon), 1, \ldots , 1)$. Since the global norm 
$\No_{K/\Q}$ is the product of the local norms at the $p$-places 
(thus identities), and since $\No_{K/\Q}(\varepsilon)=\pm (1, \ldots ,1)$, 
this yields $\iota_{\mathfrak p}(\varepsilon) = 1$ (since $N>1$), then
$\iota_p(\varepsilon) =1$ and $\varepsilon = 1$ (Leopoldt's conjecture) 
and the claim for $p \ne 2$. For $p=2$, ${\rm tor}_{\Z_2}(U_{\mathfrak p}/ 
\iota_2(E_K\otimes \Z_2) \cap U_{\mathfrak p}) = {\rm tor}_{\Z_2}
(U_{\mathfrak p}) = \mu_2$; since the image of ${\rm tor}_{\Z_2}(U_K)$
is ${\mathcal W}_K$, we have ${\mathcal I}_K \supseteq {\mathcal W}_K$
but since the image of $-1$ in $U_K$ (as global unit) is trivial, this gets
the equality ${\mathcal I}_K = {\mathcal W}_K$.
\end{proof}

The following result gives an important simplification in the context of 
Greenberg's conjecture in the totally split case \cite[Theorem 2, \S\,4]{Gree},
and explains why examples with ${\mathcal C}_{K_m} \ne 1$
will take place in the above totally split case of $p$ in $K$ 
(Theorem \ref{thmfond} and Corollary \ref{eight}):

\begin{lemma} \label{totdec}
Let $p$ be a prime and let $K=\Q(N) \subset \wh \Q^*$
(i.e., $p \nmid N$). We assume that ${\mathcal C}_K=1$. 
Let $K' \subseteq K$ be the splitting field of $p$ in $K$.

\noindent
Let $K_m:= K \Q(p^m)$, $K'_m:= K' \Q(p^m)$, $m \geq 0$. 
Then ${\mathcal C}_{K_m}=1$ if and only if ${\mathcal C}_{K'_m}=1$. Therefore,
$\lambda = \mu = \nu = 0$ if and only if $\lambda' = \mu' = \nu' = 0$ in 
terms of Iwasawa's invariants in $K_\infty$ and $K'_\infty$, respectively.
\end{lemma}

\begin{proof} Let $g:=\G(K_m/K)$; since ${\mathcal C}_K=1$, 
the Chevalley formulas become: 
$$\hbox{$\order {\mathcal C}_{K_m}^{\,g}\! = \ffrac{p^{m\, (s_p-1)}}
{(E^{\rm pos}_K : E_K^{\rm pos} \cap \No_{K_m/K} (K_m^\times))}$ and
$\order {\mathcal C}_{K'_m}^{\,g}\! = \ffrac{p^{m\, (s_p-1)}}
{(E_{K'}^{\rm pos} : E_{K'}^{\rm pos} \cap \No_{K'_m/K'} (K'^\times_m))}$}$$
since $s_p = [K' : \Q]$ is the same in the two formulas.

\smallskip
The map $E^{\rm pos}_{K'} /E_{K'}^{\rm pos} \cap \No_{K'_m/K'} (K'^\times_m)
\to E^{\rm pos}_K / E_K^{\rm pos} \cap \No_{K_m/K} (K_m^\times)$ is injective; indeed,
with obvious notations, if $\varepsilon' = \No_{K_m/K}(y)$, $y\in K_m^\times$,
then using $\No_{K_m/K'_m}$, one gets $\varepsilon'^{\,[K:K']} = \No_{K'_m/K'}(y')$,
$y' \in K'^\times_m$. The result follows since $[K:K']$ is prime to $p$. 

\smallskip
So $(E_{K'}^{\rm pos} : E_{K'}^{\rm pos} 
\cap \No_{K'_m/K'} (K'^\times_m)) \leq (E^{\rm pos}_K : E_K^{\rm pos}
\cap \No_{K_m/K} (K_m^\times))$, whence $\order {\mathcal C}_{K'_m}^{\,g}
\geq \order {\mathcal C}_{K_m}^{\,g}$; but the map ${\mathcal C}_{K'_m}
\to {\mathcal C}_{K_m}$ is injective since $[K_m : K'_m]$ is prime to $p$, 
and we get ${\mathcal C}_{K_m}^{\,g} \simeq {\mathcal C}_{K'_m}^{\,g}$.
Whence easily the claims.
\end{proof}

The following result may be considered as a corollary 
to Lemma \ref{split} (cf. \cite[Theorem 4.7]{Gra5}, \cite[Section 3]{Gra81}, 
\cite[Proposition 3.3, Theorem 1]{Gra10} for more information after the 
pioneering work of Taya \cite[Theorem 1.1]{Ta}):

\begin{lemma}\label{ta}
Let $K= \Q(N) \subset \wh \Q^*$ (i.e., $p \nmid N$) and let $K_m :=K \Q(p^m)$. 
Then the integer $\ffrac{p^{m\, (s_p-1)}}{(E^{\rm pos}_K : E_K^{\rm pos} \cap 
\No_{K_m/K} (K_m^\times) )}$ divides $\order {\mathcal R}_K^{\rm nr}$. 
If $p$ totally splits in $K$, then for all $m$ large enough there is equality with 
$\order {\mathcal R}_K^{\rm nr}=\order {\mathcal R}_K$ (cf. Diagram II).
\end{lemma}

From Lemma \ref{split}, \ref{totdec}, \ref{ta}, we get the following  
genus theory characterization, of practical use, in the framework of
the notion of $p$-rationality:

\begin{theorem} \label{thmfond}
Let $p>2$ be totally split in $K=\Q(N)$. Then, there exists $m \geq 0$, such that 
${\mathcal C}_{K_m} \ne 1$, if and only if ${\mathcal T}_K \ne 1$ (i.e., $K$
is not $p$-rational). For $p=2$, the condition becomes 
${\mathcal T}_K/{\mathcal W}_K\ne 1$.
\end{theorem}

\begin{proof} Let $g=\G(K_m/K)$.
If ${\mathcal C}_{K_m} \ne 1$ for some $m$, then ${\mathcal C}_{K_m}^g \ne 1$,
whence ${\mathcal T}_K \ne 1$ from (\ref{genres}) and Diagram II. 
Assume ${\mathcal T}_K \ne 1$; if ${\mathcal C}_K \ne 1$, then  
${\mathcal C}_{K_m}^g \ne 1$, otherwise, if ${\mathcal C}_K = 1$, then
${\mathcal R}_K \ne 1$ and for $m$ large enough, $\order {\mathcal C}_{K_m}^g 
= \order  {\mathcal G}_K = \order {\mathcal R}_K^{\rm nr} = 
\order {\mathcal R}_K = \order {\mathcal T}_K$.
\end{proof}

\begin{corollary} \label{eight}
The table \ref{table}, restricted to primes $p$ totally split in $K$, gives the following 
list of ${\mathcal C}_{K_1}\ne 1$, in the selected intervals for $N$, $p$: \par
\footnotesize
\begin{verbatim}
N=2     p=31       annihilator = Mod(1,31)*x+Mod(1,31)
N=2     p=1546463  annihilator = Mod(1,1546463)*x+Mod(1,1546463)
N=2^8   p=18433    annihilator = Mod(1, 18433)*x + Mod(9723, 18433)
N=2^10  p=114689   annihilator = Mod(1,114689)*x+Mod(66688,114689)
N=3     p=73       annihilator = Mod(1,73)*x+Mod(9,73)
N=3^4   p=487      annihilator = Mod(1,487)*x+Mod(287,487)
N=3^4   p=238627   annihilator = Mod(1,238627)*x+Mod(106366,238627)
N=5^2   p=2251     annihilator = Mod(1,2251)*x+Mod(1033,2251)
\end{verbatim}
\end{corollary}
\normalsize

In fact the literature does contain these few counterexamples (see Coates 
\cite[Section 3]{Co2}, relating results from Fukuda--Komatsu, Horie 
\cite{Ho1,Ho2,FK3,FKM}). Note that $N=3^4$, $p=238627$, needs, 
with Program XI, ${\sf time = 52,869 ms}$ and that $N=2^{10}$, $p=114689$, 
needs ${\sf time = 5min, 30,507 ms}$. We shall examine these cases 
and try to find others or to become aware of the rarity of them. 
We will also do checks, using another process, even if it's useless,
by computing Hasse's normic symbols, in $K_m/K$, of the units of $K$, using 
the ``product formula'' of class field theory, since (assuming 
${\mathcal C}_K^{\rm res}=1$), the condition ${\mathcal C}_{K_m}^{\rm res} \ne 1$ 
is equivalent to $\ffrac{p^{m\, (s_p-1)}}
{(E^{\rm pos}_K : E_K^{\rm pos} \cap \No_{K_m/K} (K_m^\times))} \ne 1$. This will be 
equivalent to the computation of the rank of a $\F_p$-matrix.

\subsection{Non-$p$-principalities in $p$-extensions}

\smallskip
Let $K=\Q(N)$ and let $p \nmid N$ totally split in $K/\Q$;
since the case $p=2$, totally split in $K$, is out of reach of the programs
we implicitly assume $p>2$ with the ordinary senses for classes and units.
Let $K_1:= K \Q(p)$; we have to compute $(E_K : E_K \cap \No_{K_1/K})$.
To avoid the instruction ${\sf bnfinit(P)}$, unfeasible for $N >17$,
we shall use the cyclotomic units \cite[Lemma 8.1\,(a)]{Wa}
giving $E_K$ assuming the base field $K$ principal,
then use the local normic Hasse's symbols.
Then, following the practical method described in \cite[II.4.4.3]{Gra0}, 
the normic symbol $(\varepsilon,K_1/K)_{\mathfrak p}$ for a unit 
$\varepsilon \in E_K$ and a ramified $p$-place ${\mathfrak p}$, 
requires to find $\alpha$ such that (the conductor being~$p^2$):
\begin{equation}\label{normic}
\alpha  \equiv \varepsilon \pmod {{\mathfrak p}^2}, \hspace{00.5cm}
\alpha \equiv 1 \pmod {(p\, {\mathfrak p}^{-1})^2}.
\end{equation}
 
Then $(\alpha)$ is an ideal, prime to $p$, whose
Artin symbol in $\G(K_1/K)$ characterizes the normic symbol;
its image in $\G(\Q(p)/\Q)$ is given by the Artin 
symbol of $\No_{K_1/\Q(p)}(\alpha)$, seen in $(\Z/p^2\Z)^\times$.

\smallskip
The program, written with $N=\ell^n$, may be generalized
using Program~XI. One must precise ${\sf el}$ and ${\sf n}$; if the rank 
is strictly less than $\ell^n-1$ (taking into account the ``product formula'')
then $(E_K : E_K \cap \No_{K_1/K}(K_1^\times)) < p^{\ell^n-1}$, 
whence ${\mathcal C}_{K_1}$ non-trivial.
The variables ${\sf m1, m2}$
denote the modulus ${\mathfrak p}^2$ and $(p\,{\mathfrak p}^{-1})^2$,
the variable ${\sf m = m1+m2}$ allows congruence (\ref{normic})
for $\alpha$ (in ${\sf Z}$). Thus one computes the $\F_p$-rank 
(in ${\sf rkM}$) of the matrix ${\sf M}$.
A sufficient precision must be chosen to compute ${\sf P}$ as irreducible
polynomial of the generating real cyclotomic unit deduced from
${\sf u=z+z^{-1}}$, where ${\sf z=exp(2*I*Pi/f)}$; thus the instruction
${\sf e=nfgaloisconj(P)}$ gives the conjugates as polynomials of ${\sf u}$.\par
\footnotesize
\begin{verbatim}
PROGRAM XIII. RANK OF THE MATRIX OF NORMIC SYMBOLS FOR el^n ODD
{el=3;n=3;N=el^n;f=el^(n+1);z=exp(2*I*Pi/f);
rho=znprimroot(f);h=rho^N;H=rho^(el-1);
P=1;for(k=1,N,c=lift(H^k);u=1;for(j=1,(el-1)/2,a=lift(c*h^j);
u=u*(z^a+z^-a));P=P*(x-u));P=round(P);e=nfgaloisconj(P);
\\p=487;\\ Choice of a special prime p or of an interval:
forprime(p=2,2*10^5,w=valuation(p^(el-1)-1,el);if(w<n+1,next);
g=znprimroot(p^2);
for(aa=1,p-1,t=norm(Mod(x-aa,P));vt=valuation(t,p);if(vt==1,a=aa;break));
A=List;for(k=1,N,listput(A,e[k]-a,k));W=List;for(j=1,N,E=Mod(e[j],P);
V=List;for(k=1,N,m1=Mod(A[k],P);m2=norm(m1)/m1;
m1=m1^2;m2=m2^2;m=m1+m2;Z=E+(1-E)*m1/m;ZZ=lift(Z);
\\This part replace Z (very huge) by a suitable integer residue:
Num=numerator(ZZ);Num0=0;for(i=0,N-1,c=polcoeff(Num,i);if(c==0,next);
v=valuation(c,p);if(v>=0,c=lift(Mod(c,p^2)));
if(v<0,c=p^v*lift(Mod(c*p^-v,p^(2-v))));
Num0=Num0+c*x^i);Num=Num0;Den=denominator(ZZ);if(Den!=1,
Den0=0;for(i=0,N-1,c=polcoeff(Den,i);if(c==0,next);
v=valuation(c,p);if(v>=0,c=lift(Mod(c,p^2)));
if(v<0,c=p^v*lift(Mod(c*p^-v,p^(2-v))));
Den0=Den0+c*x^i);Den=Den0);Z=Mod(Num,P)*Mod(Den,P)^-1;
No=Mod(norm(Z),p^2);Ln=Mod(znlog(No,g),p);
listput(V,Ln));listput(W,V));M=matrix(N,N,u,v,W[u][v]);rk=matrank(M);
if(rk<N-1,print("N=",N," p=",p," rk(M)=",rk));if(rk==N-1,
print("N=",N," control: ","p=",p," vt=",vt," root=",a," rk(M)=",rk))
)\\ End of interval p
}
PROGRAM XIV. RANK OF THE MATRIX OF NORMIC SYMBOLS FOR 2^n 
{el=2;n=4;N=el^n;f=el^(n+2);z=exp(2*I*Pi/f);H=Mod(5,f);
P=1;for(j=1,N,c=lift(H^j);u=z^(-2*c)*(1-z^(5*c))/(1-z^c);P=P*(x-u));
P=round(P);e=nfgaloisconj(P);
\\p=18433;\\ Choice of a special prime p or of an interval:
forprime(p=2,2*10^5,w=valuation(p^2-1,2);if(w<n+3,next);
g=znprimroot(p^2);
for(aa=1,p-1,t=norm(Mod(x-aa,P));vt=valuation(t,p);if(vt==1,a=aa;break));
A=List;for(k=1,N,listput(A,e[k]-a,k));W=List;for(j=1,N,E=Mod(e[j],P);
V=List;for(k=1,N,m1=Mod(A[k],P);m2=norm(m1)/m1;
m1=m1^2;m2=m2^2;m=m1+m2;Z=E+(1-E)*m1/m;ZZ=lift(Z);
\\ This part replace Z (very huge) by a suitable integer residue:
Num=numerator(ZZ);Num0=0;for(i=0,N-1,c=polcoeff(Num,i);if(c==0,next);
v=valuation(c,p);if(v>=0,c=lift(Mod(c,p^2)));
if(v<0,c=p^v*lift(Mod(c*p^-v,p^(2-v))));
Num0=Num0+c*x^i);Num=Num0;Den=denominator(ZZ);if(Den!=1,
Den0=0;for(i=0,N-1,c=polcoeff(Den,i);if(c==0,next);
v=valuation(c,p);if(v>=0,c=lift(Mod(c,p^2)));
if(v<0,c=p^v*lift(Mod(c*p^-v,p^(2-v))));
Den0=Den0+c*x^i);Den=Den0);Z=Mod(Num,P)*Mod(Den,P)^-1;
No=Mod(norm(Z),p^2);Ln=Mod(znlog(No,g),p);listput(V,Ln));
listput(W,V));M=matrix(N,N,u,v,W[u][v]);rk=matrank(M);
if(rk<N-1,print("N=",N," p=",p,"  rk(M)=",rk));if(rk==N-1,
print("N=",N," control: ","p=",p," vt=",vt," root=",a," rk(M)=",rk))
)\\ End of interval p
}
N=2  p=31      rk(M)=0     N=3    p=73  rk(M)=1     N=5^2  p=2251 rk(M)=23
N=2  p=1546463 rk(M)=0     N=3^4  p=487 rk(M)=79  
\end{verbatim}
\normalsize

For these known counterexamples, $\order {\mathcal T}_K=p$,
which indicates that $\order {\mathcal R}_K=p$ since ${\mathcal C}_K=1$.
The case $\ell=3$, $n=1$, $p=73$ may be elucidate
in more details: the units are $(\varepsilon_1=x^2+x-1, \varepsilon_2=x-1)$ 
and fulfill the relation $(\varepsilon_1^{33} \cdot \varepsilon_2^5)^{72} \equiv 
1+73^2 \cdot (2\,x^2+59\,x+69) \pmod{73^3}$,
with $2\,x^2+59\,x+69 \in {\mathfrak p} \mid 73$. Thus the inertia groups
${\rm tor}_{\Z_{73}}(U_{{\mathfrak p}_i}/\ov E_K \cap U_{{\mathfrak p}_i})$, 
$i=1,2,3$, are trivial, giving ${\mathcal R}_K^{\rm ram}=1$, 
${\mathcal R}_K^{\rm nr}={\mathcal R}_K={\mathcal T}_K$, as expected. 
In the case $\ell=5$, $n=2$, $p=2251$ totally splits in $K/\Q$; some 
computations in $E_K/E_K^{2251}$ (of order $2251^{24}$) indicate, 
as expected, that $(\varepsilon_i)^{2250} = 1+ 2251\cdot \alpha_i$, 
with non-independent $\alpha_i$ modulo $2251$, which implies,
as above, the existence of a unit local $p$th power (hence local norm), 
but not in $E_K^p$.

\smallskip
This shows that a direct $p$-adic computation on the units is hopeless contrary to
that of local norm symbols. We have performed such computations in large intervals
without finding new solutions. This enforces \cite[Conjecture~D]{Co2} in $\wh \Q$ 
and our philosophy about the $p$-rationality in general. 

\smallskip
More precisely, if one considers heuristics in the Borell--Cantelli
style, using standard probabilities~$\frac{1}{p}$, we have, possibly,
infinitely many examples, but this does not seem realistic; in 
\cite[Conjecture 8.4.]{Gra4}, we have given extensive calculations and 
justifications of an opposite situation giving, as for the well-known Fermat 
quotients of small integers $2$, $3$,... some other probabilities, for any
regulator of algebraic numbers, suggesting solutions finite in number with 
the particularity of giving very few solutions.

\subsection{Behavior of the logarithmic class groups in $\wh \Q$}\label{Clog}

The following results of Jaulent is perhaps a key to understand some 
phenomena in $\wh \Q$, regarding Greenberg's conjecture:
\begin{theorem}\label{clog} \cite[Theorem 4.5, Remarques]{Jau1}.
Let $K=\Q(N) \subset \wh \Q^*$, for some prime $p$, let $m \geq 0$
and $K_m = K \Q(p^m)$.
Under the Leopoldt and Gross--Kuz'min conjectures for $p$, 
$\wt {\mathcal C}_{K_m}=1$ if and only if $\wt {\mathcal C}_K=1$.
\end{theorem}

\begin{theorem}\cite[Th\'eor\`eme 17]{Jau3}, \cite[Th\'eor\`eme 3.4]{Gra5}.
Let $p$ be totally split in $K$. A sufficient condition to have
$\wt {\mathcal C}_K=1$ is that ${\mathcal C}_K = cl_K(S)$
and $(E_K^S : E_K^S \cap \No_{K_1/K}(K_1^\times)) = p^{N-1}$,
where $E_K^S$ is the group of $S$-units of $K$.
\end{theorem}

\begin{remark} Assume moreover that ${\mathcal C}_K=1$ 
and to simplify that ${\mathcal R}_K = p$, in other words
$(E_K : E_K \cap \No_{K_1/K}(K_1^\times)) = p^{N-2}$ (thus
$\order {\mathcal C}_{K_1}^g = p$) as are the known numerical 
examples. Then $E_K^S$ is of the form 
$E_K \plus \langle \pi_1, \ldots , \pi_N \rangle$, the $\pi_i$ being 
the generators of the ${\mathfrak p}_i \mid p$. 
This supposes that the group of norm symbols of 
$\langle \pi_1, \ldots , \pi_N \rangle$ is contained in that of 
$E_K$ (of $\F_p$-dimension $N-2$), which is of low probability.
\end{remark}

These results give many cases of triviality and we know that 
$\wt {\mathcal C}_K=1$ implies that Greenberg's conjecture holds true 
in $K_\infty$ for obvious reason. We have no counterexamples (for all 
$N < 30$ and $p \in [2, 2 \cdot 10^5]$; one may use the following program 
(give $\ell$, $n$, $p$):\par
\footnotesize
\begin{verbatim}
PROGRAM XV. COMPUTATION OF LOGARITHMIC CLASS GROUPS 
{el=3;n=1;p=73;if(el==2,P=x;for(i=1,n,P=P^2-2));
if(el!=2,P=polsubcyclo(el^(n+1),el^n));K=bnfinit(P,1);cl=K.no;
clog=bnflog(K,p);print("N=",el^n," p=",p," cl=",cl," clog=",clog)}

el=2  p=31  cl=1 clog=[[],[],[]]        el=3  p=73  cl=1 clog=[[],[],[]]
\end{verbatim}
\normalsize

So, even if for $K_1= K\,\Q(31) = \Q(2)\,\Q(31)$ ($p=31$) 
and $K_1= K \,\Q(73) = \Q(3)\,\Q(73)$ ($p=73$), the class groups 
${\mathcal C}_{K_1}$ are non-trivial, the logarithmic class groups 
$\wt {\mathcal C}_{K_1}$ are trivial.

\section{Conclusion and questions}
Genus theory (Theorem \ref{thmfond}, Corollary \ref{eight}) have succeeded to 
give few non-trivial $p$-class groups of {\it composite subfields} $\Q(pN)$ of $\wh \Q$, 
but there are not enough computations to give more precise heuristics. This invites to ask 
for some questions about the arithmetic properties of $\wh \Q$: 

\smallskip  
(i) Let $p$ be a fixed prime number. It is clear that $p$ is totally ramified in $\wh \Q/\wh \Q^*$;
thus the Frobenius of $p$ in $\wh \Q^*/\Q$ fixes a field $D_p$ such that $p$ totally splits in 
$D_p/\Q$. An out of reach question is the finiteness (or not) of $D_p$ which 
can be written $\Q( {\mathcal L}^{\mathcal N})$, ${\mathcal L} = \{\ell_1, \ldots, \ell_t, \ldots\}$, 
${\mathcal N}=\{n_1, \ldots, n_t, \ldots\}$, with an obvious meaning.
Since the number $\ell^{g_p}$ of prime ideals above $p$ in a single $\Z_\ell$-extension 
$\Q(\ell^\infty)$ is finite, the integers $n_\ell \in {\mathcal N}$ are finite but not necessarily 
${\mathcal L}$.

\smallskip
For example, if $p=2$, the only known primes $\ell$ such that $2$ splits
in part in $\Q(\ell^\infty)$ are $1093$ and $3511$; so
if there is no other case, the decomposition field of $2$ in $\wh \Q/\Q$
should be $D_2 = \Q(1093 \cdot 3511)$.

\smallskip
Is the decomposition group of $p$ in $\wh \Q/\Q$ of finite 
index in $\G(\wh \Q/\Q)$~? This is the
conjecture given in \cite[Conjecture 8.4]{Gra4}. Of course, taking a prime of the form
$p=1+ \lambda \,q_1^{a_1} \cdots q_s^{a_s}$, with primes $q_i$, $a_i \geq 2$,
gives unbounded indices since $p$ splits in $\Q(q_1^{a_1-1} \cdots q_s^{a_s-1})$.

\smallskip
(ii) Let $K = \Q(N)$ and for any $p \nmid N$, 
let $s_p = \order S$. Let $K_m=K \, \Q(p^m)$ for
$m \gg 0$ such that $\ds \frac{p^{m\, (s_p-1)}}
{(E_K^{\rm pos} : E_K^{\rm pos} \cap \No_{K_m/K} (K_m^\times))} = 
\order {\mathcal R}_K^{\rm nr}$ (Lemma \ref{ta}); is the set of 
$p$, such that ${\mathcal R}_K^{\rm nr} \ne 1$, finite in number ?
If so, this gives new feature about the units in $\wh \Q$ and is 
also related to Greenberg's conjecture in $\wh \Q$.

\smallskip
(iii) In the two cases, $\ell^n=2^8$, $p=18433 
\equiv 1 \pmod {2^{11}}$ and $\ell=2^{10}$, $p=114689  \equiv 1\pmod  {2^{14}}$, 
then ${\mathcal C}_K = 1$, ${\mathcal T}_K \ne 1$ (Programs of 
\S\,\ref{programIII-IV} and \S\,\ref{programVI-VIII}),
$p$ totally splits in $K$ and then $\order {\mathcal C}_{K_1}$
is non-trivial from Theorem \ref{thmfond}; it is divisible by $\ffrac{p^{\ell^n-1}}
{(E_K : E_K \cap \No_{K_1/K}(K_1^\times))}$ and it will be useful to verify,
but this takes too much time to compute the norm index. The computations 
have been done in \cite{AF,Ina,FKM,Fu}, but, to our knowledge, no program
is available.\,\footnote{\,I warmly thank 
Takayuki Morisawa for sending me his conference 
paper (loc.cit.), not so easy to find for me, but which contains all the bibliographical 
and numerical information that we revisit in our paper, especially summarized in 
Fukuda's lecture \cite{Fu}. The results, 
$31 \!\mid\! h(2 \cdot 31)$, $73 \!\mid\! h(3 \cdot 73)$ (Horie 2001),
$31\!\mid\! h(2 \cdot 31)$, $1546463 \!\mid\! h(2 \cdot 1546463)$, 
$73 \!\mid\! h(3 \cdot 73)$ (Fukuda, Komatsu 2011),
$18433 \!\mid\! h(2^8 \cdot 18433)$, $114689 \!\mid\! h(2^{10} \cdot 114689)$,
$487 \!\mid\! h(3^4 \cdot 487)$, $238627 \!\mid\! h(3^4 \cdot 238627)$,
$2251 \!\mid\! h(5^2 \cdot 2251)$ (Fukuda, Komatsu, Morisawa 2011),
$107 \!\mid\! h(2 \cdot 53)$ (Fukuda 2011) were announced in various articles
and conferences.}
What is the order of the logarithmic class group $\wt {\mathcal C}_K$
for these cases of too large degrees ?

\smallskip
(iv) Let $K = \Q(N)$ and $K_m=K \Q(p^m)$, for all $m \geq 0$; 
what are the Iwasawa invariants of $\ds\limproj_m\,{\mathcal T}_{K_m}$ ?

(v) In \cite{Sil} Silverman proves, after some other contributions
(Cusick, Pohst, Remak), a general inequality between $R_K$ (classical real
regulator) and $D_K$ (discriminant) of the form
$R_K > c_K ({\rm log}(\gamma_K \vert D_K\vert))^{O([K : \Q])}$.
A $p$-adic equivalent would give a solution of many questions
in number theory, as a proof of Leopoldt's conjecture\,!
However, we have proposed, in \cite[Conjecture 8.2]{Gra77} a ``folk 
conjecture'' about $\order {\mathcal T}_K$, which 
applies to ${\mathcal R}_K$, equal to ${\mathcal T}_K$ for all $p$ 
large enough, and justified by extensive computations:

\begin{conjecture}
Let ${\mathcal K}$ be the set of totally real number fields; for $K \in {\mathcal K}$,
let $D_K$ be its discriminant and let ${\mathcal R}_K := {\rm tor}_{\Z_p}
({\rm log}(U_K)/{\rm log}(\ov E_K))$ be its normalized $p$-adic 
regulator (see \S\,\ref{definitions}). There exists a constant 
$C_p > 0$ such that
${\rm log}_\infty(\order {\mathcal R}_K) \leq
{\rm log}_\infty(\order {\mathcal T}_K) \leq C_p \cdot 
{\rm log}_\infty(\sqrt {\vert D_K \vert}), \ \,\hbox{for all $K \in {\mathcal K}$},$
where ${\rm log}_\infty$ is the complex logarithm. 
Possibly, $C_p$ is independent of $p$.
\end{conjecture}

\end{document}